\documentclass{amsart}
\usepackage[OT2,T1]{fontenc}
\DeclareSymbolFont{cyrletters}{OT2}{wncyr}{m}{n}
\DeclareMathSymbol{\Sha}{\mathalpha}{cyrletters}{"58}
\usepackage{preamble}
\usepackage[T1]{fontenc}
\usepackage{comment}
\usepackage{mathtools}
 
\usepackage[margin=1.3in]{geometry}

\newcommand{\cdotspace}{\hspace{-1.5pt}\cdot\hspace{-1.5pt}}

\newcommand{\ord}{\mathrm{ord}}

\newcommand{\res}{\textup{res}}
\newcommand{\sst}{\mathrm{st}}

\newcommand{\tupH}{\textup{H}}

\definecolor{Green}{rgb}{0.0, 0.5, 0.0}

\newcommand{\p}{\mathfrak{p}}

\numberwithin{equation}{section}

\begin{document}

\title[Bianchi Stark--Heegner cycles]{Stark--Heegner cycles attached to Bianchi modular forms}

\begin{abstract}
Let $f$ be a Bianchi modular form, that is, an automorphic form for $\mathrm{GL}(2)$ over an imaginary quadratic field $F$, and let $\pri$ be a prime of $F$ at which $f$ is new. Let $K$ be a quadratic extension of $F$, and $L(f/K,s)$ the $L$-function of the base-change of $f$ to $K$. Under certain hypotheses on $f$ and $K$, the functional equation of $L(f/K,s)$ ensures that it vanishes at the central point. The Bloch--Kato conjecture predicts that this should force the existence of non-trivial classes in an appropriate global Selmer group attached to $f$ and $K$. In this paper, we use the theory of double integrals developed by Barrera Salazar and the second author to construct certain $\pri$-adic Abel--Jacobi maps, which we use to propose a construction of such classes via \emph{Stark--Heegner cycles}. This builds on ideas of Darmon and in particular generalises an approach of Rotger and Seveso in the setting of classical modular forms.
\end{abstract}

\author{Guhan Venkat}
\address{Guhan Venkat\newline
Morningside Center of Mathematics\\
Academy of Mathematics and Systems Science\\
Chinese Academy of Sciences\\
Beijing, 100190\\
China}
\email{guhan@amss.ac.cn}

\author{Chris Williams}
\address{Chris Williams\newline 
University of Warwick, Mathematics Institute, Zeeman Building, CV4 7AL, UK}
\email{christopher.d.williams@warwick.ac.uk}

\thanks{The authors' research was partially supported by the NSERC Discovery Grants Program 05710, a CRM-Laval postdoctoral fellowship (Venkat) and by a Heilbronn research fellowship (Williams).}
\maketitle

\setcounter{tocdepth}{1}
\renewcommand{\baselinestretch}{0.4}
\tableofcontents
\renewcommand{\baselinestretch}{1.0}\normalsize


\section{Introduction} \label{sec:intro}
The Birch and Swinnerton-Dyer (BSD) conjecture predicts that if $E/\QQ$ is an elliptic curve, then its analytic rank (the order of vanishing of its $L$-function at $s=1$) and algebraic rank (the $\ZZ$-rank of its group of rational points) coincide, providing a deep relationship between arithmetic and analysis. In cases where the $L$-function is known to vanish at $s=1$, it is natural to then try to construct points of infinite order on the curve. When the analytic rank is precisely 1, an example of this is given by the system of \emph{Heegner points} \cite{GZ86}, which were a key input to the celebrated proof of BSD in this case \cite{Kol88}. Extensive efforts have since been made to search for points that could play a similar role in higher-rank settings. A particularly promising avenue of research has been to use $p$-adic methods, and in \cite{Dar01} Darmon gave a construction of \emph{Stark--Heegner points}, $p$-adic points on the curve that are conjecturally global and satisfy reciprocity laws under Galois automorphisms. 

More generally, we can replace $E$ with an automorphic Galois representation $\rho$ of a number field $K$, and BSD with the \emph{Bloch--Kato conjectures}, which similarly predict that the order of vanishing of the $L$-function of $\rho$ at critical values should equal the rank of a Selmer group attached to $\rho$. Any proof of this statement is likely to hinge on the construction of explicit classes in this Selmer group in cases where the $L$-function is known to vanish. In this paper, we give a conjectural construction of such classes in the case where $\rho$ is attached to a modular form over an imaginary quadratic field $F$ (a \emph{Bianchi modular form}) and $K$ is a quadratic extension of $F$. The construction uses $p$-adic methods in the style of Darmon; in particular, we also conjecture a reciprocity law under Galois automorphisms. This generalises earlier work of Trifkovi\'c \cite{Tri06} and Rotger--Seveso \cite{RS12}, who conjectured such constructions for elliptic curves over $F$ and classical modular forms respectively.

\subsection{Set-up}
Let $p$ be a rational prime and fix embeddings $\iota_{\infty} : \overline{\QQ} \rightarrow \mathbb{C}$ and $\iota_{p}: \overline{\QQ} \rightarrow \overline{\QQ}_{p}$. Let $F$ be an imaginary quadratic field with ring of integers $\mathcal{O}_{F}$ such that $p$ is unramified in $F$, and let $\mathfrak{p}$ be the prime above $p$ in $\mathcal{O}_{F}$ corresponding to $i_{p}$. For simplicity of notation\footnote{We comment more on this assumption in \S\ref{sec:class number} below.}, let $F$ have class number one. Let $f \in S_{k,k}(U_{0}(\mathcal{N}))$ be a \emph{Bianchi cuspidal eigenform} of (parallel) even weight $k$ and square-free level $\cN = \pri\cM \subseteq \cO_{F}$, where $\mathfrak{p} \nmid \cM$, such that $f$ is new at $\pri$; see \S\ref{sec:bianchi} below for more details on our conventions. Throughout we will let $L/\Qp$ denote a (sufficiently large) coefficient field.

Let $K/F$ be a number field which is unramified at $\mathfrak{p}$ and let $V_{p}(f)$ denote the two-dimensional representation of $G_{F}:=\Gal(\overline{\QQ}/F)$ attached to $f$, taking values in a finite extension $L$ of $\QQ_p$ that contains the Hecke eigenvalues of $f$ under $\iota_p$ (see \S\ref{subsec:fontainemazur} and the references mentioned therein). This representation is conjectured to be semistable, and assuming this we consider the (semistable) \emph{Bloch--Kato Selmer group} over $K$, denoted $\textup{Sel}_{\sst}(K, V_{p}(f))$, the definition of which we recall. For every place $\nu$ of $K$, let $\tupH^{1}_{\sst}(K_{\nu}, V_{p}(f))$ be the kernel of the map \begin{equation*}  \tupH^{1}(K_{\nu}, V_{p}(f)) \rightarrow 
\begin{cases} 
\tupH^{1}(K_{\nu}^{\textup{ur}}, V_{p}(f))                    , & \text{if } \nu \nmid \mathfrak{p} \\
\tupH^{1}(K_{\nu}, B_{\sst} \otimes V_{p}(f)), & \text{if } \nu \mid \mathfrak{p}
\end{cases}
\end{equation*} where $B_{\sst}$ is Fontaine's ring of periods. The (semistable) Bloch--Kato Selmer group is then defined as 
\[\textup{Sel}_{\sst}(K, V_{p}(f)) := \ker\Big( \tupH^{1}(K, V_{p}(f)) \xrightarrow{\Pi\res_{\nu}} \frac{\tupH^{1}(K_{\nu}, V_{p}(f))}{\tupH^{1}_{\sst}(K_{\nu}, V_{p}(f))} \Big).\]

Let $\mathcal{M}_{f}$ be the (conjectural) rank 2 \emph{Clozel motive} defined over $F$ attached to the Bianchi eigenform $f$ (cf. \cite{Clo90}) and let $\textup{CH}^{\frac{k+2}{2}}(\mathcal{M}_{f} \otimes K)_{0}$ denote the Chow group of codimension $\frac{k+2}{2}$ null-homologous cycles on $\mathcal{M}_{f}$ base-changed to $K$. For each finite $\nu$, the $\mathfrak{p}$-adic Abel--Jacobi map then induces the commutative diagram
\begin{equation} \label{eqn:p-adicabeljacobi}
\begin{tikzcd}
\textup{CH}^{\frac{k+2}{2}}(\mathcal{M}_{f} \otimes K)_{0} \arrow{d}{\res_{\nu}} \arrow{r}{\textup{cl}^{\frac{k+2}{2}}_{1,K}} & \textup{Sel}_{\sst}(K, V_{p}(f)(\frac{k+2}{2})) \arrow{d} \\
\textup{CH}^{\frac{k+2}{2}}(\mathcal{M}_{f} \otimes K_{\nu})_{0} \arrow{r}{\textup{cl}^{\frac{k+2}{2}}_{1,K_{\nu}}} & \tupH^{1}_{\sst}(K_{\nu}, V_{p}(f)(\frac{k+2}{2})).
\end{tikzcd}
\end{equation} 
The Beilinson--Bloch--Kato conjectures then predict that the morphism $\textup{cl}_{1,K}^{\frac{k+2}{2}}$ is an isomorphism and that
\[ 
\textup{dim}_{L}\left(\textup{Sel}_{\sst}\left(K, V_{p}(f)\left(\tfrac{k+2}{2}\right)\right)\right) = \text{ord}_{s=\frac{k+2}{2}}L(f/K,s), 
\]
where $L(f/K,s)$ is the $L$-function of the base-change of $f$ to $K$. 

For the rest of this article we will fix a quadratic extension $K/F$ of relative discriminant $\mathcal{D}_{K}$ prime to $\mathcal{N}$. In this situation, the sign of the functional equation of $L(f/K,s)$ is given by $(-1)^{\#S(\mathcal{N},K)}$ (see \cite[Eqn. (2)]{Vat04}), where 
\[ S(\mathcal{N},K) := \{\text{finite places }\nu : \nu\text{ is inert in }K \}. \]

We will further assume that $K$ satisfies the \emph{Stark--Heegner hypothesis}, i.e.
\begin{itemize}
\item $\mathfrak{p}$ is inert in $K$,
\item all primes $\mathfrak{l} \mid \cM$ split in $K$. 
\end{itemize}

In particular, the Stark--Heegner hypothesis forces $S(\mathcal{N},K) = \{\pri\}$, hence forces the sign of the functional equation to be $-1$, and hence forces the vanishing of the central critical $L$-value of the base-change $L$-function, i.e.\ $L(f/K, \frac{k+2}{2}) = 0$. 

This allows us to force higher orders of vanishing over ring class extensions. Let $\mathcal{C} \subseteq \mathcal{O}_{F}$ be any ideal relatively prime to $\mathcal{N}\mathcal{D}_{K}$ and let 
\[ \mathcal{O}_{\mathcal{C}} := \mathcal{O}_{F} + \mathcal{C}\mathcal{O}_{K} \]
be the \emph{$\mathcal{O}_{F}$-order of conductor} $\mathcal{C}$ in $K$. By global class field theory, we know that the Galois group $G_{\mathcal{C}} := \Gal(H_{\mathcal{C}}/K)$ is isomorphic to the Picard group $\textup{Pic}(\mathcal{O}_{\mathcal{C}})$, where $H_{\mathcal{C}}/K$ is the ring class field of conductor $\mathcal{C}$. For any character $\chi : G_{\mathcal{C}} \rightarrow \mathbb{C}^{\times}$, the sign of the functional equation of the twisted $L$-function $L(f/K, \chi, s)$ is again $-1$. Moreover, the $L$-function admits a factorisation
\[ L(f/H_{\mathcal{C}},s) = \prod\limits_{\chi \in G_{\mathcal{C}}^{\vee}} L(f/K,\chi,s) \]    
over the characters of $G_{\mathcal{C}}$, and it follows that
\[ \textup{ord}_{s=\frac{k+2}{2}}L(f/H_{\mathcal{C}},s) \geq h(\mathcal{O}_{\mathcal{C}}) := |G_{\mathcal{C}}|. \] 
The Bloch--Kato conjecture predicts the existence of a family of non-trivial cohomology classes
\[ 
\bigg\{ s_{\mathcal{C}} \in \textup{Sel}_{\sst}\bigg(H_{\mathcal{C}}, V_{p}(f)\left(\tfrac{k+2}{2}\right)\bigg) \bigg\} \]
over towers of class fields $H_{\mathcal{C}}/K$ for $\mathcal{C}$ relatively prime to $\mathcal{N}\mathcal{D}_{K}$. The main aim of this paper is a proposal of conjectural candidates for such a family of classes via \emph{Stark--Heegner cycles}.

\subsection{Stark--Heegner cycles}
To construct our proposed classes, we build a machine that takes as input certain homology classes, the Stark--Heegner cycles, and outputs classes in a local Selmer group. We first describe the input to the machine.

Let $\mathcal{R}$ be the Eichler $\mathcal{O}_{F}[1/\mathfrak{p}]$-order in $\textup{M}_{2}(\mathcal{O}_{F}[1/\mathfrak{p}])$ of matrices that are upper triangular modulo $\mathcal{M}$ and let $\Gamma := \mathcal{R}_{1}^{\times}$ be the set of invertible matrices of $\mathcal{R}$ of determinant $1$. The starting point for our constructions is the space 
\[
(\Delta_{0} \otimes \textup{Div}(\mathcal{H}_{\mathfrak{p}}^{\textup{ur}}) \otimes V_{k,k})_{\Gamma},
\]
 where $\Delta_{0} := \textup{Div}^{0}(\mathbb{P}^{1}(F))$, $\textup{Div}(\mathcal{H}_{\mathfrak{p}}^{\textup{ur}})$ denotes the subgroup of divisors supported on $\mathbb{P}^{1}(\mathbb{Q}_{p}^{\textup{ur}})\setminus\mathbb{P}^{1}(F_{\mathfrak{p}})$ that are fixed by the action of ${\textup{Gal}(\mathbb{Q}_{p}^{\textup{ur}}/L^{0})}$ (where $L^{0} := L \cap \mathbb{Q}_{p}^{\textup{ur}}$), and $V_{k,k} := V_{k} \otimes V_{k}$ where $V_{k}$ is the ring of homogenous polynomials of degree $k$ in two variables with coefficients in $L$. This space should be regarded as an explicit substitute for the local Chow group. 
 
The \emph{Stark--Heegner} (or \emph{Darmon}) cycles are a family of elements
\[ 
\mathrm{D}_{\Psi} \in (\Delta_{0} \otimes \textup{Div}(\mathcal{H}_{\mathfrak{p}}^{\textup{ur}}) \otimes V_{k,k})_{\Gamma} 
\]
attached to optimal embeddings $\Psi : \mathcal{O}_{\mathcal{C}} \hookrightarrow \mathcal{R}$, where $\mathcal{O}_{\mathcal{C}}$ is an $\mathcal{O}_{F}[1/\mathfrak{p}]$-order in $K$ of conductor $\mathcal{C}$ prime to $\mathcal{N}\mathcal{D}_{K}$. A precise definition is contained in \S\ref{sec:stark-heegnercycles}.

\subsection{Monodromy modules and $\pri$-adic Abel--Jacobi maps}
 We describe the machine, which broadly breaks into two pieces: an analytic half, via $\pri$-adic Abel--Jacobi maps from cycles to monodromy modules; and an algebraic half, where we use $p$-adic Hodge theory and the Bloch--Kato exponential map to pass from  monodromy modules to Selmer groups.

 	\subsubsection{$\pri$-adic integration}
 	In \S\ref{sec:integration}, we develop the analytic half of this machine. We use a theory of $\mathfrak{p}$-adic integration inspired by Teitelbaum's construction in \cite{Teitelbaum1990}. We remark that our approach, using results of \cite{BW19}, follows the ``modular symbol theoretic'' $\mathfrak{p}$-adic integration of Seveso (see \cite{Sev1}) rather than the more general, but slightly less explicit, ``cohomological theoretic'' $\mathfrak{p}$-adic integration of Rotger and Seveso pursued in \cite{RS12}. 
 
\paragraph*{} Using the fact that $f$ is $\pri$-new, the results of \cite{BW19} allow us to extend the overconvergent Bianchi modular symbols of \cite{Wil17} over the Bruhat--Tits tree $\mathcal{T}_{\pri}$ for $\GLt/F_{\pri}$, giving a space which we denote $\mathbf{MS}_\Gamma(L)$. Our integration theory yields pairings
\[\Phi^{\mathrm{log}_{p},\sigma}, \Phi^{\mathrm{ord}_{\pri}} :  (\Delta_{0} \otimes \textup{Div}^{0}(\mathcal{H}_{\mathfrak{p}}^{\textup{ur}}) \otimes V_{k,k})_{\Gamma} \ \otimes \ \mathbf{MS}_\Gamma(L) \longrightarrow L\]
associated to each embedding $\sigma:F_{\pri}\hookrightarrow L$. Inspired by this, in \S4 we construct a filtered Frobenius monodromy module $\mathbf{D}_f \in \textup{MF}(\varphi, N, F_{\pri}, L)$, the category of filtered Frobenius monodromy modules over $F_{\pri}$ with coefficients in $L$. Fixing an embedding $\sigma : F_{\pri}\hookrightarrow L$, our $\mathfrak{p}$-adic integration can then be interpreted as a map
\[ 
\Phi_{f}\defeq -\Phi_{f}^{\log_{p},\sigma}\oplus\Phi^{\ord_{\pri}}_{f} : (\Delta_{0} \otimes \textup{Div}^{0}(\mathcal{H}_{\mathfrak{p}}^{\textup{ur}}) \otimes V_{k,k})_{\Gamma} \rightarrow \mathbf{D}_{f,L},
\]
where $\mathbf{D}_{f,L} := \mathbf{D}_f \otimes_{F_{\pri} \otimes L} L$. Following \cite{Sev1} and \cite{Dar01},  in Theorem~\ref{thm:padicabeljacobiimageofdarmoncycles} we lift the above morphism to the \emph{$\pri$-adic Abel--Jacobi map}
\[ \Phi^{\textup{AJ}} : (\Delta_{0} \otimes \textup{Div}(\mathcal{H}_{\mathfrak{p}}^{\textup{ur}}) \otimes V_{k,k})_{\Gamma} \rightarrow \mathbf{D}_{f,L}/\textup{Fil}^{\frac{k+2}{2}}(\mathbf{D}_{f,L}), \]
 removing the degree zero condition on $\mathrm{Div}(\uhp_{\pri}^{\mathrm{ur}})$. 
 \begin{remark}
 	Throughout the construction, we work only at a particular prime $\pri$ above $p$. When $p$ is split, this means we use the ``partially overconvergent'' coefficients of \cite[\S6]{Wil17}: that is, we use coefficients that are overconvergent at $\pri$, but classical at $\pribar$. This means we do not require any conditions on $f$ at the conjugate $\pribar$.
 \end{remark}

\subsubsection{Fontaine--Mazur theory and Bloch--Kato exponential maps}
The second half of the machine is more algebraic in nature. Recalling the assumption that $V_{p}(f)$ is semistable, let $\mathbb{D}_{f} := \mathbb{D}_{\textup{st}}(V_{p}(f))$ be the admissible filtered $(\varphi,N)$ module over $F_{\mathfrak{p}}$ attached to the local Galois representation $V_{p}(f)|_{G_{F_{\mathfrak{p}}}}$. The bridge between the analytic and algebraic theory is (conjecturally) provided by the \textit{trivial zero conjecture} (Conjecture~\ref{conj:trivialzeroconjecture}), that predicts an equality of analytic and algebraic $\mathcal{L}$-invariants, which would give an isomorphism 
\[
	\frac{\mathbf{D}_{f,L}}{\mathrm{Fil}^{\frac{k+2}{2}}(\mathbf{D}_{f,L})} \cong \frac{\mathbb{D}_{f,L}}{\mathrm{Fil}^{\frac{k+2}{2}}(\mathbb{D}_{f,L})}
\]
as explained in Theorem~\ref{thm:identifcationofmonodromymodules} and (\ref{eqn:identifcationofmonodromymodules}). Here $\mathbb{D}_{f,L} := \mathbb{D}_{f} \otimes_{F_{\pri} \otimes L} L$. Further, we have an identification
\begin{equation} \label{eqn:blochkato}
	\frac{\mathbb{D}_{f,L}}{\textup{Fil}^{\tfrac{k+2}{2}}(\mathbb{D}_{f, L})} \cong \textup{H}^{1}_{\textup{st}}\left(L, V_{p}(f)\big(\tfrac{k+2}{2}\big)\right) 
\end{equation}
afforded by the \emph{Bloch--Kato exponential map}. In particular, assuming the conjectural equality of $\mathcal{L}$-invariants, the complete machine is given by the composition
\begin{align} \label{eqn:abel-jacobi-intro}
(\Delta_{0} \otimes \textup{Div}(\mathcal{H}_{\mathfrak{p}}^{\textup{ur}}) \otimes V_{k,k})_{\Gamma} \xrightarrow{\Phi^{\textup{AJ}}} &\frac{\mathbf{D}_{f,L}}{\textup{Fil}^{\tfrac{k+2}{2}}(\mathbf{D}_{f,L})}\notag \\
 \cong &\frac{\mathbb{D}_{f,L}}{\textup{Fil}^{\tfrac{k+2}{2}}(\mathbb{D}_{f,L})} \cong \textup{H}^{1}_{\textup{st}}\left(L, V_{p}(f)\big(\tfrac{k+2}{2}\big)\right). 
\end{align}

\begin{remark}
	Note that the Galois structure we conjecture on the analytic monodromy module, via the trivial zero conjecture, is a natural one to consider. In particular, the main result of \cite{BW19} related the analytic $\mathcal{L}$-invariant we consider in this paper to the arithmetic of the standard $L$-function of the Bianchi form $f$, that is, the $L$-function of its standard Galois representation $V_p(f)$, to which the algebraic $\mathcal{L}$-invariant should be attached. Moreover the conjecture is true for base-change forms. We comment further on this in \S\ref{sec:evidence} below.
\end{remark}

\subsection{Global conjectures for Stark--Heegner cycles}

Let $s_{\Psi} \in \textup{H}^{1}_{\textup{st}}(L, V_{p}(f)(\frac{k+2}{2}))$ be the image of the Stark--Heegner cycle $\mathrm{D}_{\Psi}$ under the map (\ref{eqn:abel-jacobi-intro}). Assuming $L \supseteq K_{\pri}$, we have an inclusion $\iota_p : H_{\mathcal{C}} \hookrightarrow L$ (since $\mathfrak{p}$ splits completely in $H_{\mathcal{C}}$, the ring class field for $\mathcal{O}_{\mathcal{C}}$) that induces 
\begin{equation} \label{eqn:restriction}
	\textup{res}_{\mathfrak{p}} : \textup{Sel}_{\sst}\bigg(H_{\mathcal{C}}, V_{p}\left(\tfrac{k+2}{2}\right)\bigg) \rightarrow \tupH^{1}_{\sst}\bigg(L, V_{p}\left(\tfrac{k+2}{2}\right)\bigg). 
\end{equation}
In \S\ref{sec:stark-heegnercycles}, we formulate conjectures (see Conjecture~\ref{conj:mainconjecture}) that predict:
\begin{itemize}
\item[(a)] the existence of global classes $\mathcal{S}_{\Psi} \in \textup{Sel}_{\sst}(H_{\mathcal{C}}, V_{p}(\frac{k+2}{2}))$ such that \[
	\textup{res}_{\mathfrak{p}}(\mathcal{S}_{\Psi}) = s_{\Psi},
\] 
\item[(b)] and a \emph{Shimura reciprocity law} describing the Galois action of $G_{\mathcal{C}}$ on these global classes.
\end{itemize}
 These are direct analogues of conjectures of Darmon on Stark--Heegner points on elliptic curves over the rationals \cite{Dar01}. Further, we formulate Gross--Zagier type formulae (Conjecture~\ref{conj:grosszagier}) relating these classes to the derivative of the base-change $L$-function $L(f/K,s)$ at $s=\frac{k+2}{2}$.   

\subsection{Evidence}\label{sec:evidence}
Our main conjecture itself relies on the validity of two other conjectures. We briefly comment on theoretical and computational evidence for each of these, which lead us to believe Conjecture \ref{conj:mainconjecture} is a reasonable one to make.

The first conjecture we require is the semistability of the standard Galois representation $V_p(f)$ attached to the Bianchi modular form $f$. This is a standard conjecture that is widely believed and which may be within reach in the near future given recent developments in the study of such Galois representations. Moreover it is true for base-change forms, which can be deduced from the validity of the analogous result for classical modular forms. 

The second is the trivial zero conjecture for Bianchi modular forms, or more precisely an equality of analytic and algebraic $\mathcal{L}$-invariants attached to $V_p(f)$. In the base-change case, we prove this in Lemma~\ref{lem:basechangetrivialzero}. In the general case, this conjecture directly generalises a conjecture of Trifkovi\'c in weight 2, which is backed by extensive computational evidence when $p$ is split \cite{Tri06,GMS15}. Further, we can consider the analogous statements over totally real fields, where semistability is known. Over $\Q$, the various $\mathcal{L}$-invariants attached to the usual Galois representation of a modular form are all known to be equal (see, for example, \cite{BDI10}). Similarly, in the case of Hilbert modular forms, the direct analogue of our analytic $\mathcal{L}$-invariant was constructed in \cite{Spi14,Ber17}, and is known to be equal to the algebraic one \cite{Mok09,BDJ17}.

Finally, we comment on Conjecture \ref{conj:mainconjecture}. It would be very interesting to obtain computational evidence towards this conjecture, though it appears difficult to make the machine outlined above, and the Selmer group in which our classes live, completely explicit. Despite this, there are special cases that have been studied previously. If, for example, $E/F$ is a modular elliptic curve corresponding to a weight 2 Bianchi modular form $f$, then the Selmer group can be described much more concretely via the group of rational points on $E$. In this case, our conjecture degenerates to that of Trifkovi\'c \cite{Tri06}, for which again there is significant computational evidence \cite{GMS15}.

\subsection{Remarks on higher class number}\label{sec:class number}
The assumption that $F$ has class number 1 is made only to simplify notation. Indeed, the main reason we assume it is that it was present throughout in the results of \cite{BW19}, which we use crucially in \S\ref{sec:integration}. As explained in \S11.1 \emph{op.\ cit}., this assumption was made because it significantly simplified the notation. The methods of that paper are classical (rather than adelic); thus in the case of class number $h$, one must consider $h$ copies of the Bruhat--Tits tree, and the direct sum of $h$ copies of modular symbols, and the Hecke operators at non-principal primes permute these summands. The higher class number case would be better treated adelically as in G\"{a}rtner \cite{Gar12}.

In the present paper, beyond the dependence on \cite{BW19}, the only places where we use the class number 1 assumption are in Harder's  Eichler--Shimura isomorphism \eqref{eqn:eichlershimuraharder} and in \S\ref{sec:construction of homology classes}, where we fix a generator of the discriminant. Both could be dealt with by instead considering a direct sum of modular symbol spaces over the class group, on which non-principal ideals act by permutation. This approach -- at the cost of notational strain -- was the one taken, for example, in \cite{Wil17}.

\subsection{Comparison to relevant literature}
Our conjectures should be viewed as adding to a significant body of work in this direction. We have already mentioned the work of Trifkovi\'c and Rotger--Seveso above, both themselves natural generalisations of work of Darmon \cite{Dar01}. In other settings, similar constructions of Stark--Heegner/Darmon points have been proposed by several authors, on Modular Jacobians, Shimura curves, and elliptic curves over arbitrary number fields; see for instance \cite{Das05,Gre09,GSS16,GM14,GMS15,LRV12}.

\subsection*{Acknowledgements} It will be evident to the reader that we are greatly indebted to the mathematics of Henri Darmon. We also thank him for his comments on an earlier version of this draft. This project got underway at the BIRS-CMO workshop \textit{Special values of automorphic $L$-functions and associated $p$-adic $L$-Functions}, and continued at \textit{Iwasawa 2019}, where the authors were participants. We would like to thank Marco Seveso for explaining in detail his constructions in \cite{RS12,Sev1} which served as motivation for this project, and Denis Benois, Mladen Dimitrov, Mahesh Kakde, David Loeffler, A. Raghuram and Giovanni Rosso for their invaluable comments. Finally, we heartily thank the anonymous referee for numerous remarks and suggestions on an earlier version.

\section{Preliminaries and Bianchi modular forms} \label{sec:bianchi}

\subsection{Notation}

For an even integer $k \geq 0$, and a ring $R$, we define $V_k(R)$ to be the space of polynomials of degree $\leq k$ over $R$. This space is naturally a left $\GLt(R)$-module via
\[
	\left(\smallmatrd{a}{b}{c}{d}\cdotspace P\right)(x) = \tfrac{(a+cx)^k}{(ad-bc)^{k/2}}P\left(\tfrac{b+dx}{a+cx}\right).
\]
The factor of determinant means the centre acts trivially, and the action descends to $\PGLt(R)$.
\begin{remark} \label{rem:PGL(2)action}
Note the factor of the determinant makes sense only when the weight $k$ is even (such as the Bianchi modular forms considered in this paper). This factor was not considered in \cite{Wil17} which allows all integral weights $k$. The effect of twisting by the $k/2$-th power of the determinant is seen only on the Hecke operators which are scalar multiples of those considered in \cite{Wil17}.
\end{remark}
%
%
\subsection{Bianchi modular forms and their $L$-functions}
\label{sec:bianchi modular forms}
Bianchi modular forms are adelic automorphic forms for $\GLt$ over the imaginary quadratic field $F$. 
Their study goes back decades, and we recall only the essentials; for a fuller account in our conventions, see \cite[\S1]{Wil17}. For any open compact subgroup $U$ of $\GLt(\A_F^f)$, and any $k \ge -1$, there is a finite-dimensional $\C$-vector space $S_{k,k}(U)$ of \emph{Bianchi cusp forms} of weight $(k, k)$ and level $U$, which are functions
\[ f : \GLt(F) \backslash \GLt(\A_F) / U \longrightarrow V_{2k+2}(\C) \]
transforming appropriately under the subgroup $\C^\times \cdotspace \ \SUt(\C)$ and that satisfy suitable harmonicity and growth conditions. We specialise to the case where $k \geq 0$ and 
\[
	U = U_0(\cN) = \big\{\smallmatrd{a}{b}{c}{d} \in \GLt(\widehat{\roi}_F) : c \equiv 0 \newmod{\cN}\big\},
\]
where $\cN = \pri\cM \subset \roi_F$ is square-free and $\pri \nmid \cM$.

These forms admit an analogue of $q$-expansions (cf.~\cite[\S1.2]{Wil17}), giving rise to a system of Fourier--Whittaker coefficients $c(I,f)$ indexed by ideals $I \subset \mathcal{D}^{-1}$ (where $\mathcal{D}$ is the different of $F/\Q$). These can be described as the eigenvalues of Hecke operators: indeed, there is a family of (commuting) Hecke operators indexed by ideals $\m \subset \roi_F$, defined by double coset operators, and if $f$ is a Hecke eigenform normalised so that $c(1,f) = 1$, then the eigenvalue $\lambda_{\m}$ of the $\m$-th Hecke operator on $f$ is equal to $c(\m\mathcal{D}^{-1},f)$ (see \cite[Cor.\ 6.2]{Hid94}).

If $M$ is a suitable module equipped with an action of the Hecke operators, and $f$ is a cuspidal Bianchi eigenform, we write $M_{(f)}$ for the $f$-isotypic part of $M$, that is, the generalised eigenspace where the Hecke operators act with the same eigenvalues as on $f$. 
\begin{definition}
Let $\Lambda(f,\varphi)$ denote the (completed) $L$-function of $f$, normalised so that if $\varphi$ is a Hecke character of $F$ of infinity type $(q,r)$, where $q,r \gg 0$, then
\[
\Lambda(f,\varphi) = \frac{\Gamma(q+1)\Gamma(r+1)}{(2\pi i)^{q+r+2}}\sum_{\m \subset \roi_F, I\neq 0} c(\m,f)\varphi(\m) N(\m)^{-1}.
\]
This admits an analytic continuation to all such characters.
\end{definition}
The `critical' values of this $L$-function can be controlled; in particular, by \cite[Thm.\ 8.1]{Hid94}, we see that there exists a period $\Omega_{f} \in \C^\times$ and a number field $E$ such that, if $\varphi$ is a Hecke character of infinity type $0 \leq (q,r) \leq (k,k)$, with $q,r \in \Z$, we have
\begin{equation}\label{period}
\frac{\Lambda(f,\varphi)}{\Omega_{f}} \in E(\varphi),
\end{equation}
where $E(\varphi)$ is the number field over $E$ generated by the values of $\varphi$. (Note that $\Omega_f$ is only well-defined up to $E^\times$. We will not be concerned with questions of integrality in this paper).

\subsection{Base-change}
Let $f$ be a Bianchi modular form as above. We will be fundamentally interested in the \emph{base-change} of $f$ to a quadratic extension $K/F$. This is not strictly an operation on the automorphic \emph{form} $f$, but rather on the automorphic \emph{representation} $\pi$ of $\GLt(\A_F)$ that it generates. Let $V_{p}(f)$ be the $p$-adic Galois representation of $G_F$ attached to $f$. 

\begin{definition}
For a field $E$, let $W_E$ denote the Weil group of $E$. An automorphic representation $\pi_K$ of $\GLt(\A_K)$ is said to be the \emph{base-change} of $\pi$ if for each place $v$ of $F$, and each $w|v$ of $K$, the Langlands parameter attached to $\pi_{K,w}$ is the restriction to $W_{K_w}$ of the Langlands parameter $\sigma_v : W_{F_v} \rightarrow \GLt(\C)$ of $\pi_v$. On Galois representations, this corresponds to restriction from $G_F$ to $G_K$ of $V_{p}(f)$.
\end{definition}

\begin{theorem}[R.P. Langlands, \cite{Langlands1}]
Every cuspidal automorphic representation $\pi$ of $\GLt(\A_F)$ has a unique base-change lift $\pi_K$ to $\GLt(\A_K)$. It is cuspidal unless $\pi$ has CM by $K$, that is, $\pi$ is dihedral, induced from a Hecke character of $K$.
\end{theorem}

\subsection{Rankin--Selberg $L$-functions} \label{subsec:rankin-selberg}
Let $\pi$ be the cuspidal automorphic representation of $\GL_{2}(\mathbb{A}_{F})$ generated by $f$ and let $\chi$ be a (unitary) Hecke character of $K$ trivial on $\mathbb{A}_{F}^{\times}$. Denote by $\pi_{\chi}$ the automorphic representation $\mathrm{Ind}_{K}^F\chi$ on $\GL_{2}(\mathbb{A}_{F})$ induced by $\chi$. Let $L(\pi,\chi,s)$ denote the \textit{Rankin-Selberg} $L$-function associated to the $\chi$-twist of $\pi$. Then it is known that this $L$-function has a meromorphic continuation to the entire complex plane and satisfies the functional equation
\[ L(\pi,\chi,s) = \epsilon(\pi,\chi,s)L(\pi^{\vee},\chi^{-1},1 - s), \]
where $\pi^{\vee}$ denotes the contragradient dual of $\pi$ and $\epsilon(\pi,\chi,s)$ is the $\epsilon$-factor. The condition that $\chi|_{\mathbb{A}_{F}^{\times}} = 1$ implies that $L(\pi^{\vee},\chi^{-1},1 - s) = L(\pi,\chi,1-s)$ (cf. \cite[\S1.1]{CV07}) and the functional equation thus yields
\[ L(\pi,\chi,s) = \epsilon(\pi,\chi,s)L(\pi,\chi,1 - s). \] 
The parity of the order of vanishing at the central value $s = 1/2$ is thus determined by the value of 
\[ \epsilon(\pi,\chi) := \epsilon(\pi,\chi,1/2) \in \{\pm 1\}. \]
In this paper, we will be primarily interested in collection of \textit{ring class characters} $\chi$ of varying conductor, the definition of which we recall.
\begin{definition}\label{defn:ringclasscharacters}
Let $\mathcal{C}$ be any $\mathcal{O}_{F}$-ideal. A ring class character $\chi$ of conductor $\mathcal{C}$ is a (unitary) Hecke character that factors through the finite group 
\[ K^{\times}\backslash\mathbb{A}_{K}^{\times}\slash K_{\infty}^{\times}\widehat{\mathcal{O}}_{\mathcal{C}}^{\times} \cong \textup{Pic}(\mathcal{O}_{\mathcal{C}}) \]
where $K_{\infty} := K \otimes_{\mathbb{Q}} \mathbb{R}$ are the infinite adeles, $\mathcal{O}_{\mathcal{C}} := \mathcal{O}_{F} + \mathcal{C}\mathcal{O}_{K}$, and $\widehat{\cO}_{\cC} := \cO_{\cC}\otimes \widehat{\Z}$.
\end{definition}
In this paper, we will only consider ring class characters of conductor $\mathcal{C}$ prime to $\mathcal{N}\mathcal{D}_{K}$.
Note that $\epsilon(\pi,\chi) := \prod\limits_\nu \epsilon(\pi_{\nu},\chi_{\nu}),$ where $\nu$ runs through the places of $F$ and $\epsilon(\pi_{\nu},\chi_{\nu})$ is the local sign attached to the local components of $\pi$ and $\pi_{\chi}$ at the place $\nu$. We let
\[ S(\chi) := \{ \nu | \epsilon(\pi_{\nu},\chi_{\nu}) = -1 \}. \]
Whilst the various formulas for the local $\epsilon$-factors are spread throughout \cite{JL70} and \cite{Jac72}, we can provide a much simpler description of the set of places in $S(\chi)$ in our setting, noting that the infinite place is split (see \cite{Vat04} for a detailed exposition). In fact, we have
\[ S(\chi) = S(\mathcal{N},K) := \{ \nu | \nu \textup{ is finite and inert in K and } \textup{ord}_{\nu}(\mathcal{N}) \text{ is odd} \}. \]
In particular, the sign of the Rankin--Selberg $L$-function depends only on $K/F$ and the level $\mathcal{N}$ of $f$, and not on the ring class character $\chi$. Further, the \emph{Stark--Heegner} hypothesis implies that $|S(\chi)| = 1$ is odd so that $\epsilon(\pi,\chi) = -1$.

\subsection{Arithmetic subgroups and modular symbols}
We now pass to a more algebraic description of the relevant automorphic forms.

\begin{definition}
Let $\Delta_0 \defeq \mathrm{Div}^0(\Proj(F))$ be the space of degree zero divisors on the cusps. This admits an action of $\PGLt(F)$ by M\"{o}bius transformations $x \mapsto (ax+b)/(cx+d)$. 
\end{definition}
For any ring $R$, let $V_{k,k}(R) \defeq V_{k}(R)\otimes_R V_k(R)$, which has an action of $\PGLt(R)^2$ via the action on each component. If $L$ is a field containing the image of both embeddings $F \hookrightarrow \overline{\Q}$, then this gives an action of $\PGLt(F)$ on $V_{k,k}(L)$, with the action on the first component via one embedding and on the second component via its conjugate.

Now let $\Gamma \subset \PGLt(F)$ be any subgroup, and $V$ a right $\Gamma$-module. Define
\[
	\Delta(V) \defeq \mathrm{Hom}(\Delta_0,V),
\]
which we equip with an action of $\Gamma$ by
\[
	(\gamma \cdotspace \phi)(D) \defeq \phi(\gamma D)|\gamma.
\]
We define the space of \emph{$V$--valued modular symbols for $\Gamma$} to be the $\Gamma$-invariant maps
\[
	\Symb_{\Gamma}(V) \defeq \h^0(\Gamma,\Delta(V)).
\]

\begin{remark}
For the congruence subgroup $\Gamma_0(\cN) := U_0(\cN) \cap \SLt(\cO_F) \subset \SLt(\cO_F)$, recalling that $\cN = \pri\cM$, the space of \emph{Bianchi modular symbols} is the space $\Symb_{\Gamma_0(\cN)}(V_{k,k}^\vee(\Cp))$, where we equip $V_{k,k}^\vee$ with the (right) dual action.
This space admits an action of the Hecke operators. Moreover, since $F$ has class number one, there is a Hecke-equivariant injection due to the Eichler--Shimura isomorphism of G{\"u}nter Harder (see \cite{Har87})
\begin{equation} \label{eqn:eichlershimuraharder} S_{k,k}(U_0(\cN)) \hookrightarrow \Symb_{\Gamma_0(\cN)}(V_{k,k}^\vee(\C))
\end{equation}
and the cokernel consists of Eisenstein eigenpackets. In particular, to each Bianchi modular eigenform $f$ there exists an attached eigensymbol $\phi_f$. This symbol was crucial to the constructions of \cite{BW19}. For more details, see \cite{Wil17}.
\end{remark}

\subsection{Overconvergent modular symbols}
We recap the overconvergent modular symbols of \cite{Wil17}, which are fundamental to our theory of $\pri$-adic integration. When $p$ is inert, we use the full module of overconvergent modular symbols. When $p$ splits as $\pri\overline{\pri}$, we instead use the partially overconvergent coefficients defined in \S6 \emph{op.\ cit}, and hence impose no conditions at $\pribar$. Everything in this section is explained in greater detail in \cite[\S3,\S7]{Wil17}.

\begin{definition}
Let $F_{\pri}$ be the completion of $F$ at $\pri$ and $\cO_{\pri}$ its ring of integers. For a complete field extension $L/\Qp$, let $\cA_k(\cO_{\pri},L)$ (resp. $\cA_k(\cO_{p},L)$) denote the ring of locally analytic functions $\cO_{\pri} \rightarrow L$ (resp. $\cO_{p} \defeq \cO_{F} \otimes_{\ZZ} \ZZ_{p} \rightarrow L)$. Then:
\begin{itemize}
\item[(i)] If $p$ is inert in $F$, let $\cA^{\pri}_{k,k}(\cO_p,L )= \cA_k(\cO_{\pri},L)$;
\item[(ii)] If $p\cO_F = \pri\pribar$ splits, let \[
	\cA_{k,k}^{\pri}(\cO_p,L) = \cA_k(\cO_\pri,L) \widehat{\otimes} V_k(\cO_{\pribar},L),
\]
where $V_k(\cO_{\pribar},L)$ is the ring of polynomial functions $\cO_{\pribar} \rightarrow L$ of degree $\leq k$.
\end{itemize}
\end{definition}
Note that $\cA^{\pri}_{k,k}(\cO_{\pri},L ) \subseteq \cA_k(\cO_{p},L)$. Moreover $\cA^{\pri}_{k,k}(\cO_{\pri},L)$ carries a natural weight $k$ left action of the semigroup
\[
	\Sigma_0 := \left\{\smallmatrd{a}{b}{c}{d} \in M_2(\cO_p) : v_{\pri}(c) > 0, v_{\pri}(a) = 0, ad-bc \neq 0\right\}
\]
given by
\begin{equation}\label{eq:sigma action}
	\smallmatrd{a}{b}{c}{d}\cdotspace f(z) = \tfrac{(a+cz)^k}{(ad-bc)^{k/2}}f\left(\tfrac{b+dz}{a+cz}\right).
\end{equation}
\begin{remark}
	We note that we ask only that $v_{\pri}(c) > 0$, and put no restriction on $v_{\pribar}(c)$. If we let $\Sigma_0' \subset \Sigma_0$ be the subgroup where also $v_{\pribar}(c) > 0$, then $\Sigma_0'$ acts on $\cA_{k}(\cO_p,L)$ by \eqref{eq:sigma action}, and this action preserves the subspace $\cA^{\pri}_{k,k}(\cO_p,L)$.
\end{remark}
We define $\cD_{k,k}^{\pri}(\cO_p,L)$ to be the continuous dual of $\cA_{k,k}^{\pri}(\cO_p,L)$ with the dual right (weight $k$) action. Note that the inclusion $V_{k,k} \subset \cA_{k,k}^{\pri}$ (compatible under the weight $k$ left actions defined above) induces dually a surjection $\cD_{k,k}^{\pri} \rightarrow V_{k,k}^\vee,$ which in turn induces a Hecke-equivariant map
\[
\rho : \Symb_{\Gamma_0(\cN)}(\cD_{k,k}^{\pri}(\cO_p,L)) \longrightarrow \Symb_{\Gamma_{0}(\cN)}(V_{k,k}^\vee(L)).
\]

\begin{theorem}\label{thm:control theorem}
We have
\[
	\rho|_{(f)} : \Symb_{\Gamma_0(\cN)}(\cD_{k,k}^{\pri})_{(f)} \cong \Symb_{\Gamma_{0}(\cN)}(V_{k,k}^\vee)_{(f)},
\]
that is, the restriction of $\rho$ to the $f$-isotypic subspaces of the Hecke operators is an isomorphism. In particular, there is a unique overconvergent lift $\psi_f$ of $\phi_f$ under $\rho$.
\end{theorem}

\begin{proof}
In the inert case, this is \cite[Cor. 5.9]{Wil17}, noting that as $f$ is new at $\pri$, we have $v_{\pri}(c(\pri,f)) = k < k+1$ (see e.g.\ \cite[Cor.\ 4.8]{BW19}). In the split case, this follows more or less immediately from the results of \cite[\S6]{Wil17} noting that $v_{\pri}(c(\pri,f)) = k/2 < k+1;$ in any case, this locally analytic, partially overconvergent formulation is \cite[Lem.\ 9.8]{BW16}.
\end{proof}

%
%
\section{$\p$-adic integration and an $\mathcal{L}$-invariant} \label{sec:integration}

In this section we recall some results on the Bruhat--Tits tree and $\pri$-adic integration from \cite{BW19}. One of the key technical results \emph{op.\ cit}.\ was the construction of a canonical family of distributions on the projective line attached to $f$, which were used to define double integrals in the style of Darmon. We use these double integrals in a slightly different context.

\subsection{The Bruhat--Tits tree}
One of the main themes of \cite{BW19} was the `spreading out' of the Bianchi modular symbol $\phi_f$, which is defined using invariance under $\Gamma_0(\cN)$, to a function that is invariant under the action of a larger group $\Gamma$, where we allow non-integral coefficients at $\pri$. This was done via the Bruhat--Tits tree.
\begin{definition}
Let $\pri$ be a prime of $F$ above $p$ and let $\pi_{\pri}$ be a uniformiser of $\cO_{\pri}$. We denote by $\treep$ the Bruhat--Tits tree for $\GLt(F_{\pri})$, that is, the connected tree whose vertices are homothety classes of lattices in a two-dimensional $F_{\pri}$-vector space $V$. Two vertices $v$ and $v'$ are joined by an edge if and only if there are representatives $L$ and $L'$ respectively such that
\[\pi_{\pri}L' \subset L \subset L'.\]
Each edge comes equipped with an orientation. Denote the set of (oriented) edges of $\treep$ by $\mathcal{E}(\treep)$ and the set of vertices by $\mathcal{V}(\treep).$
\end{definition}
Define the \emph{standard vertex} $v_*$ to be the vertex corresponding to the lattice $\roi_{\pri}\oplus\roi_{\pri}$, and the \emph{standard edge} $e_*$ to be the edge joining $v_*$ to the vertex corresponding to $\roi_{\pri}\oplus\pi_{\pri}\roi_{\pri}$. There is a natural notion of \emph{distance} between two vertices $v$ and $v'$, and we say a vertex $v$ is \emph{even} (resp.\ \emph{odd}) if the distance between $v$ and $v_*$ is even (resp.\ odd). Each (oriented) edge has a source and a target vertex, and we say such an edge is even if its source is. Write $\vertices^+(\treep)$ and $\edges^+(\treep)$ (resp.\ $\vertices^-(\treep)$ and $\edges^-(\treep)$) for the set of even (resp.\ odd) vertices and edges resp.

There is a natural transitive action of $\PGLt(F_{\pri})$ on $\treep$, which we extend to a larger group.
\begin{definition}\label{group omega}Recall $\cN = \pri\cM$ with $\pri\nmid \cM$.
\begin{itemize}
\item[(i)] For $v$ a finite place of $F$, define 
\[R_0(\cM)_v \defeq \left\{\matr\in\mathrm{M}_2(\roi_v): c \equiv 0 \newmod{\cM}\right\}.\]
\item[(ii)] Let $R = R_0(\cM) \defeq \left\{\gamma \in \mathrm{M}_2\left(\A_F^f\right): \gamma_{v} \in R_0(\cM)_{v}\text{ for }v\neq \pri, \gamma_{\pri}\in\mathrm{M}_2(F_{\pri}) \right\}.$
\item[(iii)]Let $\widetilde{U}$ denote the image of $R^\times$ in $\mathrm{PGL}_2\left(\A_F^f\right)$.
\item[(iv)] Let $U \defeq \mathrm{PGL}_2^+\left(\A_F^f\right)\cap\widetilde{U},$ where 
\[\PGLt^+\left(\A_F^f\right) \defeq \left\{\gamma \in \PGLt\left(\A_F^f\right):v_{\pri}(\det(\gamma_{\pri})) \equiv 0\newmod{2}\right\}.\]
\item[(v)] Finally, let 
\[
	\widetilde{\Gamma} =\widetilde{U}\cap\PGLt(F), \hspace{15pt} \Gamma = U\cap\PGLt(F).
\]
\end{itemize}
\end{definition}
These groups act on $\treep$ via projection to $\mathrm{PGL}_2(F_{\pri})$. 
\begin{proposition}\label{prop:transitive}
\begin{itemize}
\item[(i)]The group $\widetilde{U}$ acts transitively on $\vertices(\treep)$ and $\edges(\treep)$. 
\item[(ii)]
The group $U$ acts transitively on $\edges^\pm(\treep)$ and $\mathcal{V}^\pm(\treep)$.
\end{itemize}
\end{proposition}
\begin{proof}
See \cite{Ser80}, Theorem 2 of Chapter II.1.4.
\end{proof}

The edges of $\treep$ are naturally in bijection with open sets in $\Proj(F_{\pri})$; for an edge $e$, write $e = \gamma_e e_*$ using (i) above. The corresponding open set is
\[
\cU(e) \defeq \gamma_e^{-1}(\roi_\pri) = \left\{x\in\Proj(F_{\pri}): \gamma_e x \in \roi_{\pri}\right\} \subset \Proj(F_{\pri}),
\]
where we let $\PGLt(F_{\pri})$ act on $\Proj(F_{\pri})$ by M\"{o}bius transformations.

\subsection{Bianchi forms on the tree}
Let $R$, $\widetilde{U}$ and $U$ be as above.
\begin{definition}
We say a function $f: \edges(\treep) \rightarrow V_{2k+2}(\C)$ is \emph{harmonic} if
\[f(\overline{e}) = -f(e) \hspace{12pt}\forall e\in\edges(\treep)\]
and if for all vertices $v\in\vertices(\treep)$, we have
\[\sum_{e: s(e) = v}f(e) = \sum_{e: t(e) = v}f(e) = 0.\]
\end{definition}
Denote by $\PGLt^f(\A_F) \defeq \GLt(\C)\times\PGLt(\A_F^f)$, and note that $U$ acts on $\PGLt^f(\A_F)$ by right multiplication. We define Bianchi cusp forms on $\treep$ as follows. 
\begin{definition}
A \emph{cusp form} on $\treep\times\PGLt^f(\A_F)$ of weight $(k,k)$ for $U$ is a function
\[\mathcal{F}:\edges(\treep)\times\PGLt^f(\A_F) \longrightarrow V_{2k+2}(\C)\]
such that:
\begin{itemize}
\item[(i)] $\mathcal{F}(\gamma e,g\gamma) = \mathcal{F}(e,g)$ for all $\gamma \in U$.
\item[(ii)] $\f$ is harmonic as a function on $\edges(\treep)$.
\item[(iii)] For each edge $e \in \edges(\treep)$, the function
\[\mathcal{F}_{e}(g) \defeq \mathcal{F}(e,g)\]
is a cusp form of weight $(k,k)$ and level $U_{e} \defeq \mathrm{Stab}_U(e).$
\end{itemize}
Denote the space of such forms by $\s_{k,k}(U,\treep)$.
\end{definition}

\begin{theorem}[\cite{BW19}, Thm.\ 2.7]\label{new forms tree}
Let $\cN = \pri\cM$. The association $\f \mapsto \f_{e_*}$ defines an isomorphism
\[\s_{k,k}(U,\treep) \cong S_{k,k}(U_0(\cN))^{\pri-\mathrm{new}}.\]
\end{theorem}

\subsection{Distributions arising from the tree}

The tree allows us to extend distributions from $\cO_p$ to be `projective in $\pri$'.
\begin{definition}
Let $\cA_k(\PP^1(F_{\pri}),L)$ denote the space of $L$-valued functions on $\PP^1(F_{\pri})$ that are locally analytic except for a pole of order at most $k$ at $\infty$. Then:

\begin{itemize}
\item[(i)] If $p$ is inert in $F$, define $\PP^1_{\pri} = \PP^1(F_{\pri})$, and $\cA_k^{\pri}(\PP^1_{\pri},L) = \cA_k(\PP^1(F_{\pri}),L)$. 
\item[(ii)] If $p\cO_F = \pri\pribar$ in $F$, define $\PP^1_{\pri} = \PP^1(F_{\pri}) \times \roi_{\pribar}$, and let 
\[
\cA_{k}^{\pri}(\PP^1_{\pri},L) = \cA_k\big(\PP^1(F_{\pri}),L\big)\widehat{\otimes}V_k\big(\cO_{\pribar},L\big).
\]
\end{itemize}
\end{definition}

Let $\cD_k^{\pri}(\PP^1_{\pri},L) := \mathrm{Hom}_{\mathrm{cts}}(\cA_k^{\pri}(\PP^1_{\pri},L), L)$. This space of distributions is a right $\Gamma$-module, and there is a natural restriction map 
\[
	\cD_k^{\pri}(\PP^1_{\pri},L) \longrightarrow \cD_{k,k}^{\pri}(\cO_p,L),
\]
inducing
\[
	\rho_{\cT} : \Symb_{\Gamma}(\cD_{k}^{\pri}(\PP^1_{\pri},L)) \longrightarrow \Symb_{\Gamma_0(\cN)}(\cD_{k,k}^{\pri}(\cO_p,L)).
\]
The domain of this map should be considered as `overconvergent Bianchi modular symbols on $\cT_{\pri}$', and carries a natural action of the Hecke operators making the map $\rho_{\cT}$ Hecke equivariant. The Theorem below is an overconvergent analogue of Theorem \ref{new forms tree}.

\begin{theorem} \label{thm:overconvergentprojectivemodularsymbol}
Let $f \in S_{k,k}(U_0(\cN))$ be a cuspidal Bianchi eigenform that is new at $\pri$. 
Then 
\[
\rho_{\cT}|_{(f)} : \Symb_{\Gamma}(\cD_{k}^{\pri}(\PP^1_{\pri},L))_{(f)} \cong \Symb_{\Gamma_0(\cN)}(\cD_{k,k}^{\pri}(\cO_p,L))_{(f)},
\] 
that is, the restriction of $\rho_{\cT}$ to the $f$-isotypic Hecke-eigenspaces is an isomorphism.
\end{theorem}

In particular, combining with Theorem \ref{thm:control theorem}, attached to $f$ we obtain a canonical element $\mu_{f} \in \Symb_{\Gamma}(\cD_k^{\pri}(\PP^1_{\pri},L))$. Explicitly, we recall that this is a system of distributions
\[
	\mu_f\{r - s\} \in \cD_k^{\pri}(\PP^1_{\pri},L),
\]
where $r,s \in \Proj(F)$, with $\mu_f\{r-s\} + \mu_f\{s - t\} = \mu_f\{r-t\}$ and $\mu_f\{\gamma r - \gamma s\}|\gamma = \mu_f\{r - s\}$ for all $\gamma \in \Gamma = U \cap \PGLt(F)$.

\begin{proof} 
This is a reformulation of \S5 of \cite{BW19}, which we briefly sketch. First we show surjectivity; let $\psi \in \Symb_{\Gamma_0(\cN)}(\cD_{k,k}^{\pri}(\cO_p,L))_{(f)}$. Following Propositions 5.4 (when $p$ is inert) and 5.8 ($p$ split) \emph{op.\ cit}.\ we again `propagate' $\psi$ over the tree to construct a lift of $\psi$ under $\rho_{\cT}$. This uses the correspondence between the edges of the Bruhat--Tits tree and basic open sets in $\PP^1(F_{\pri})$, which allows the definition of a system of distributions $\mu\{r-s\}$ defined on $\Gamma$-translates of $\cO_p$. To extend to a distribution on all of $\PP^1_{\pri}$, we crucially use that $f$ is new at $\pri$ to glue these together over the whole tree (c.f. the proof of Proposition 5.4 \emph{op.\ cit.}).

To see injectivity, suppose $\mathrm{res}(\mu) = 0$, and $\mu \neq 0$. Then there exist $r,s\in \PP^1(F)$ and $g \in \cA_k^{\pri}(\PP_{\pri}^1,L)$ such that $\mu\{r-s\}(g) \neq 0$. Let $\{e_i : i \in I\}$ be the set of edges in $\cT_{\pri}$ with source $v_*$; then $\PP_{\pri}^1 = \bigsqcup_i \cU(e_i)$, and there exist $\gamma_i \in \Gamma$ such that $\cU(e_i) = \gamma_i^{-1}\cdot\cO_p.$ Then $\gamma_{i}\cdot g \in \cA(\cO_p,L)$ has support on $\cO_p$. In particular, we have
	\begin{align*}
		\mu\{r-s\}(g) = \sum_{i \in I}\mu\{r-s\}\big(g|_{\cU(e_i)}\big) &= \sum_{i \in I}\mu\{\gamma_{i}r - \gamma_{i}s\}(\gamma_{i}\cdotspace g)\\
			&= \sum_{i\in I}\mathrm{res}(\mu)\{\gamma_{i}r - \gamma_{i}s\}(\gamma_{i}\cdotspace g) = 0,
	\end{align*}
	which is a contradiction; hence the map is injective.
\end{proof}
For brevity, we set $\mathbf{MS}_{\Gamma}(L) \defeq \textup{Symb}_{\Gamma}(\cD_k^{\pri}(\PP^1_{\pri},L)).$

\begin{remark}\label{rem:multiplicityone}
There is an identification of modular symbols with degree one compactly supported cohomology (see \cite[Lemma 8.2]{BW19}). By multiplicity one for Bianchi forms (see \cite[Proposition 3.1]{Hid94}) and the Eichler--Shimura--Harder isomorphism, we know that the $f$-isotypic part of this cohomology group is one-dimensional, providing the coefficient field $L$ contains the Hecke field of $f$. From Theorems \ref{thm:control theorem} and \ref{thm:overconvergentprojectivemodularsymbol}, it follows that $\mathbf{MS}_{\Gamma}(L)_{(f)}$ is a one-dimensional $L$-vector space.
\end{remark}

\subsection{Double integrals}

We now define double integrals in the style of \cite{Dar01}.
We fix forever a choice of $p$-adic logarithm $\log_p : \Cp^\times \longrightarrow \Cp$ such that $\log_{p}(p) = 0$. Recall that $F_{\pri}$ is the completion of $F$ at the prime $\pri$ above $p$ induced by $\iota_p$. Let $\sigma \in \Gal(F_\pri/\Q_p)$ be an automorphism, corresponding to a choice of embedding of $F_{\pri} \hookrightarrow L$. We have:
	\begin{itemize}
		\item[(i)] If $p$ is split, then $F_{\pri} = \Qp$ and $\sigma = \sigma_{\mathrm{Id}}$ must be the identity;
		\item[(ii)] If $p$ is inert, then either $\sigma = \sigma_{\mathrm{Id}}$ or $\sigma = \sigma_{\mathrm{Fr}}$, the lift of the Witt vector Frobenius to $F_{\pri}$ (that is, the non-trivial Galois automorphism).
	\end{itemize}  
\begin{definition} \label{defn:doubleintegrals}
	Let $x,y \in \uhp_{\pri}^{\mathrm{ur}}$, let $P \in V_{k,k}(\Cp)$ and let $r,s \in \Proj(F)$.
\begin{itemize}
\item[(i)] Define the `log' double integral at $\sigma$ by
	\[
	\int_x^y\int_r^s (P)\omega_{\mu}^{\textup{log},\sigma} \defeq \int_{\Proj_{\pri}}\log_p\left(\frac{t_{\pri}-x}{t_{\pri}-y}\right)^{\sigma}P(t)d\mu\{r-s\}(t),
\]
	where $t_{\pri}$ is the projection of $t \in \Proj_{\pri}$ to $\Proj(F_{\pri})$. (Note that this is well-defined since $\uhp_{\pri}^{\mathrm{ur}} \defeq \Proj(\QQ_{p}^{\mathrm{ur}})\backslash\Proj(F_{\pri})$, so that $t_{\pri}-x$ and $t_{\pri}-y$ cannot vanish and their quotient gives an element of $(\QQ_{p}^{\mathrm{ur}})^\times$);
\item[(ii)] Define the `ord' double integral by
\begin{align*}
\int_x^y\int_r^s (P)\omega_{\mu}^{\textup{ord}} &\defeq \int_{\Proj_{\pri}}\mathrm{ord}_{\pri}\left(\frac{t_{\pri}-x}{t_{\pri}-y}\right)P(t)d\mu\{r-s\}(t)\\
&=\sum\limits_{e : \textup{red}_{\mathfrak{p}}(x) \rightarrow \textup{red}_{\mathfrak{p}}(y)} \int_{\cU(e)} P(t)d\mu\{r-s\}(t),
\end{align*} 
	where $\mathrm{red}_{\pri}:\uhp_{\pri} \rightarrow \mathcal{T}_{\pri} = \mathcal{E}(\mathcal{T}_{\pri}) \sqcup \mathcal{V}(\mathcal{T}_{\pri})$ is the reduction map and $\cU(e)$ is the corresponding open set of $\Proj_{\pri}$.
\end{itemize}
Here we normalise so that $\mathrm{ord}_{\pri}(p) = 1$, noting that $p$ is a uniformiser in $F_{\pri}$.
\end{definition}
\begin{remark} \label{rem:normedlog}
There is a subtle error in \cite{BW19} in the case $p$ inert: the `log' integral of Def.~6.1 \emph{op.\ cit}.\ would be written $\int\int(P)\omega_\mu^{\mathrm{log},\sigma_{\mathrm{Id}}}$ in the notation of Def.~\ref{defn:doubleintegrals}. However, to properly obtain the connection to $p$-adic $L$-values in \S7.3 \emph{op.\ cit}, it should instead have been the sum
		\[
	\int_x^y\int_r^s (P)\omega_{\mu}^{\textup{log},\sigma_{\mathrm{Id}}} + \int_x^y\int_r^s (P)\omega_{\mu}^{\textup{log},\sigma_{\mathrm{Fr}}} = \int_{\Proj_{\pri}}\log_p\circ\  \mathrm{Norm}_{F_{\pri}/\QQ_{p}}\left(\frac{t_{\pri}-x}{t_{\pri}-y}\right)P(t)d\mu\{r-s\}(t) \]
  (compare \cite[\S5]{Spi14}). Having made this change, the rest of the results in the paper go through without modification. This is also the conceptual reason why the scalar 2 appears in \cite[Prop.~10.2]{BW19}: there is one copy of $\mathcal{L}_p(\widetilde{f})$ for each element of $\mathrm{Gal}(F_{\pri}/\Q_p)$.
\end{remark}
\begin{remark} \label{rem:additiveintegrals}
In \cite{Dar01}, Darmon defines \emph{multiplicative} double integrals attached to classical weight 2 modular forms, and by taking log and ord of the values of this integral, one obtains formula close to the above, justifying the terminology and seemingly disparate definitions. In our setting, the multiplicative double integral itself does not exist, but we can still make sense of its additive counterparts.
\end{remark}
\begin{remark} \label{rem:unramifieduhp}
The second equality in Defn.~\ref{defn:doubleintegrals}(ii) follows from \cite[Lemma 2.5]{BDG04}. Recall that $\uhp_{\pri}^{\mathrm{ur}} \defeq \Proj(\QQ_{p}^{\mathrm{ur}})\backslash\Proj(F_{\pri})$. The reduction map $\mathrm{red}_{\pri}:\uhp_{\pri}^{\mathrm{ur}} \rightarrow \mathcal{T}_{\pri}$ takes values in $\mathcal{V}(\mathcal{T}_{\pri})$ and not in $\mathcal{E}(\mathcal{T}_{\pri})$ (See \cite[\S 2]{BDG04} and the remarks after \cite[Defn.~ 3.8]{RS12}).
\end{remark}
\begin{definition}\label{def:logp}
	For $?$ either $(\textup{log}_{p},\sigma)$ or $(\textup{ord}_{\pri})$, define maps
\begin{align*}
\Phi^{?} : &[\Delta_{0} \otimes \textup{Div}^{0}(\mathcal{H}_{\mathfrak{p}}^{\textup{ur}}) \otimes V_{k,k}] \otimes 	\textup{Hom}(\Delta_{0}, \cD_k(\PP^1_{\pri},L)) \longrightarrow L\\
	&\big[(r - s)\ \otimes \ (x - y) \ \otimes \ \ P \big]\ \otimes \ \mu \ \ \ \  \ \ \ \longmapsto \int_x^y\int_r^s (P)\omega_{\mu}^{?}.
\end{align*}
In light of Remark~\ref{rem:normedlog}, we also define
\[ \Phi^{\textup{log}_{p}} = \Phi^{\textup{log}_p\circ N_{F_\pri/\Qp}} \defeq \sum_{\sigma\in\mathrm{Gal}(F_{\pri}/\Qp)} \Phi^{\textup{log}_{p},\sigma}. \] 
\end{definition}

Since the double integrals $\Phi^{\log_{p},\sigma}$ and $\Phi^{\ord_{\pri}}$ are $\Gamma$-invariant, they
induce pairings
\[ \Phi^{\log_p,\sigma},\ \Phi^{\ord_{\pri}} : (\Delta_{0} \otimes \textup{Div}^{0}(\uhp_{\mathfrak{p}}^{\textup{ur}}) \otimes V_{k,k})_{\Gamma} \otimes \mathbf{MS}_{\Gamma}(L) \rightarrow L, \]
or equivalently as morphisms
\[ \Phi^{\log_p,\sigma},\ \Phi^{\ord_{\pri}} : (\Delta_{0} \otimes \textup{Div}^{0}(\uhp_{\mathfrak{p}}^{\textup{ur}}) \otimes V_{k,k})_{\Gamma} \rightarrow \mathbf{MS}_{\Gamma}(L)^{\vee}. \]
Taking the usual short exact sequence $0 \rightarrow \mathrm{Div}^0 \rightarrow \mathrm{Div} \rightarrow \Z \rightarrow 0$ and tensoring with the flat $\Z$-module $\Delta_0\otimes V_{k,k}$, we obtain an exact sequence
\begin{equation} \label{eqn:divisorexactsequence}
0 \rightarrow \Delta_0 \otimes \textup{Div}^{0}(\uhp_{\pri}^{\textup{ur}}) \otimes V_{k,k} \rightarrow \Delta_0 \otimes \textup{Div}(\uhp_{\pri}^{\textup{ur}}) \otimes V_{k,k} \rightarrow \Delta_0 \otimes V_{k,k} \rightarrow 0,
\end{equation}
and on taking $\Gamma$-homology, we get a connecting morphism
\begin{equation}\label{eqn:homologyconnecting}
\tupH_{1}(\Gamma, \Delta_{0} \otimes V_{k,k}) \xrightarrow{\delta} (\Delta_0 \otimes \textup{Div}^{0}(\uhp^{\textup{ur}}_{\pri}) \otimes V_{k,k})_{\Gamma}.
\end{equation}
Note that via the universal coefficient theorem, we have an identification
\[ 
\tupH_{1}(\Gamma, \Delta_{0} \otimes V_{k,k}) = \tupH^{1}(\Gamma, \Delta(V_{k,k}^\vee))^{\vee} ,
\]
recalling that $\Delta(V) = \mathrm{Hom}(\Delta_0,V)$.
There is a natural action of the Hecke operators on
$\tupH_{1}(\Gamma, \Delta_{0} \otimes V_{k,k})$ making this identification Hecke-equivariant, so that we have an identification
\[
	\tupH_{1}(\Gamma, \Delta_{0} \otimes V_{k,k})_{(f)}=\tupH^{1}(\Gamma,\Delta(V_{k,k}^\vee))_{(f)}^{\vee}.
\]
By \cite[Thm.\ 8.6]{BW19}, the right-hand side is one-dimensional; hence we deduce that $\tupH_{1}(\Gamma, \Delta_{0} \otimes V_{k,k})_{(f)}$ is a one-dimensional $L$-vector space.
\begin{theorem}\label{thm:H1H0} The morphism
\[
\textup{pr}_{f}\circ\Phi^{\ord_{\pri}}\circ\delta : \tupH_{1}(\Gamma, \Delta_{0} \otimes V_{k,k}) \rightarrow \mathbf{MS}_{\Gamma}(L)_{(f)}^{\vee} 
\]
is a surjective map of $L$-vector spaces, where $\textup{pr}_{f}$ is the projection onto the $f$-isotypic component. In particular, after extending scalars to $\Cp$, it induces a Hecke-equivariant isomorphism 
\[  
\tupH_{1}(\Gamma, \Delta_{0} \otimes V_{k,k}(\Cp))_{(f)} \cong \mathbf{MS}_{\Gamma}(\Cp)_{(f)}^{\vee}. 
\]
\end{theorem}
\begin{proof}
Since the target is one-dimensional by Remark \ref{rem:multiplicityone}, it suffices to prove that the map is non-zero. For this we use the following; see \cite{BW19}, Corollary 7.4.
\begin{lemma} \label{lem:rohrlich}
Let $\chi$ be a finite order Hecke character of $F$ of prime-to-$p$ conductor such that $\chi(\pri) = \omega$, where $c(\pri,f) = -\omega N(\pri)^{k/2}$. Then for each $v \in (\cO_F/c)^\times$ there exist $\gamma_{c,v} \in \Gamma$ and $P_{c,v} \in V_{k,k}$ such that
\[
	\left[\sum_{v \in (\cO_F/c)^\times} \chi(v)\Phi^{\mathrm{ord}_{\pri}}\bigg((v/c - \infty) \otimes (\tau - \gamma_{c,v}\tau) \otimes P_{c,v}\bigg)\right] (\mu_f) = C(\chi)\Lambda(f,\chi,k/2+1),
\]
for some $\tau \in \uhp_\pri$ and where $C(\chi)$ is an explicit non-zero scalar.
\end{lemma}

By a theorem of Rohrlich (in the introduction of \cite{Roh91}), this central critical $L$-value is non-zero for all but finitely many $\chi$, and hence it is possible to choose $\chi$ satisfying the conditions such that the right-hand side is non-zero. The left-hand side is $\Phi^{\mathrm{ord}_{\pri}}$ evaluated at the element 
\[
	S_\chi := \sum_{v\in(\cO_F/c)^\times} \chi(v)\bigg[(v/c-\infty)\otimes (\tau - \gamma_{c,v}\tau) \otimes P_{c,v}\bigg].
\] 
Evaluating at $\mu_f$ factors through the projection to the $f$-isotypic part, so $\mathrm{pr}_f \circ \Phi^{\mathrm{ord}_{\pri}}$ is non-zero.

To complete the proof of surjectivity, it remains to prove that the element $S_\chi$ is in the image of $\delta$. Since this is the connecting map in the long exact sequence of homology, this is equivalent to proving that $S_\chi$ is in the kernel of
\[
	(\Delta_0\otimes \mathrm{Div}^0(\uhp^{\mathrm{ur}}_{\pri}) \otimes V_{k,k})_\Gamma \longrightarrow (\Delta_0\otimes \mathrm{Div}(\uhp^{\mathrm{ur}}_{\pri}) \otimes V_{k,k})_\Gamma.
\]
We do this termwise. From the construction of \cite[Defn.~ 7.1]{BW19}, both $(v/c - \infty)$ and $P_{c,v}$ are invariant\footnote{There is a typo here in \cite[Defn.~7.1]{BW19}, where it incorrectly states that $-v/c$ is fixed instead of $v/c$; but this is written correctly later in the paper.} under the action of $\gamma_{c,v} \in \Gamma$, and hence 
\[
	[(v/c-\infty)\otimes \tau \otimes P_{c,v}] = [(v/c-\infty)\otimes \gamma_{c,v}\tau \otimes P_{c,v}]
\]
in the $\Gamma$-coinvariants, and each term of the sum is in the kernel, as required.

For the second claim, note that both $\Phi^{\mathrm{ord}_{\pri}}$ and the maps in the long exact sequence of homology are Hecke-equivariant, and hence $\mathrm{pr}_f \circ \Phi^{\mathrm{ord}_{\pri}}\circ \delta$ is too. Hence it factors through the $f$-isotypic quotient of $\mathrm{H}_1$. We are left with a surjective map
\[  
\tupH_{1}(\Gamma, \Delta_{0} \otimes V_{k,k})_{(f)} \longrightarrow \mathbf{MS}_{\Gamma}(L)_{(f)}^{\vee} 
\]
of 1-dimensional vector spaces, hence the claim.
\end{proof}

Theorem \ref{thm:H1H0} above allows us to recover the Darmon--Orton style cohomological $\mathcal{L}$-invariant defined by Daniel Barrera Salazar and the second author in \cite{BW19}.
\begin{corollary}\label{cor:L-invariant} 
For each embedding $\sigma:F_{\pri}\hookrightarrow L$, there exists a unique $\mathcal{L}^{\sigma}_{\pri} \in \Cp$ such that
\[
	\Phi^{\log_p,\sigma}_{f}\circ\delta = \mathcal{L}^{\sigma}_{\pri}\circ\Phi^{\ord_{\pri}}_{f}\circ\delta : \tupH_{1}(\Gamma,\Delta_{0}\otimes V_{k,k}) \rightarrow \mathbf{MS}_{\Gamma}(L)_{(f)}^{\vee},
\]
where $\Phi^{*}_{f} := \textup{pr}_{f}\circ\Phi^{*}$ for $*\in\{(\log_p,\sigma);\ord_{\pri}\}$. For each prime $\pri$, we have an equality
\[
	\mathcal{L}_{\pri}^{\mathrm{BW}} = \sum_{\sigma} \mathcal{L}_{\pri}^\sigma,
\]
where $\sigma$ ranges over all embeddings and $\mathcal{L}_{\pri}^{\mathrm{BW}}$ is the $\mathcal{L}$-invariant of \cite{BW19}.
\end{corollary}
\begin{proof}
From Theorem~\ref{thm:H1H0}, we know these maps factor through the one-dimensional $f$-isotypic quotient of $\h_1$; and any two such maps must differ by some $\cL_{\pri}^\sigma \in \Cp.$ The uniqueness follows from surjectivity of $\Phi^{\mathrm{ord}_{\pri}}_f\circ\delta.$ To see the second part, note that taking the sum shows that $\Phi_f^{\mathrm{log}_p}\circ \delta = (\sum_\sigma \mathcal{L}_{\pri}^\sigma) \Phi_{f}^{\mathrm{ord}_p} \circ \delta$, for $\Phi_f^{\mathrm{log}_p}$ as in Def.~\ref{def:logp}. To show that this sum is indeed equal to the $\cL$-invariant of \cite{BW19}, it suffices to check this on any value; so we choose any element mapping to $S_\chi$ (as above) under $\delta$. The results of \cite{BW19} (and Remark \ref{rem:normedlog}) show that $\Phi^{\mathrm{log}_{p}}_f(S_\chi) = \cL_{\pri}^{\mathrm{BW}}\Phi^{\ord_{\pri}}_f(S_\chi)$, from which we deduce the result.
\end{proof}
\begin{remark}\label{rem:vectorvaluedL-invariants}
For $p$ inert in $F$, we have a vector-valued $\cL$-invariant $\vec{\cL}_{\pri}\defeq(\cL_{\pri}^{\sigma_{\mathrm{Id}}},\cL_{\pri}^{\sigma_{\mathrm{Fr}}}) \in (\mathbb{C}_{p})^{2}$. 
\end{remark}

\begin{remark}
	We note that it should be possible to formulate a purely algebraic proof of Theorem \ref{thm:H1H0}, without appealing to the deep non-vanishing results of Rohrlich, by generalising the approach of \cite[Thm.\ 3.11]{RS12}. This result -- which is equivalent to ours -- involves instead interpreting the $\cL$-invariant as a Hecke-equivariant endomorphism on the new subspace of modular forms. This in particular means they do not project to the $f$-isotypical space, but only to the new subspace, denoting this projection $\mathrm{pr}_{\mathrm{c}}$. The analogue of their result in our setting would be: there exists a unique endomorphism $\cL_{\Gamma}^\sigma \in \mathrm{End}_{\mathbb{T}_\Gamma}(\mathbf{MS}_{\Gamma}(L)_{\mathrm{new}}^\vee)$ such that 
	\[
			 \mathrm{pr}_{\mathrm{c}} \circ \Phi^{\mathrm{log}_{p},\sigma} \circ \delta = \cL_\Gamma^\sigma \circ \mathrm{pr}_{\mathrm{c}} \circ \Phi^{\mathrm{ord}_{\mathfrak{p}}} \circ \delta.
	\]
	Here $\mathbf{MS}_{\Gamma}(L)_{\mathrm{new}}$ is the new subspace of $\mathbf{MS}_{\Gamma}(L))$, $\mathrm{pr}_{\mathrm{c}}$ is projection on to this space, and $\mathbb{T}_\Gamma$ is an appropriate Hecke algebra acting on this space. If $L$ contains the Hecke field of $f$, then $\cL_\Gamma^\sigma$ preserves the (1-dimensional) $f$-isotypical subspace, upon which it acts by the scalar $\cL_{\pri}^\sigma$ of Corollary \ref{cor:L-invariant}.In this case, $\cL_{\pri}^\sigma \in L.$
\end{remark}

\section{Filtered $(\varphi, N)$-modules} \label{sec:monodromy}
\begin{definition}\label{defn:monodromymodules}
A filtered $(\varphi,N)$ module of rank $r$ over $F_{\pri}$ with coefficients in $L$ is a rank~$r$ $(F_{\pri} \otimes_{\QQ_p} L)$-module $D$ with the following additional structure:
\begin{itemize}
\item \emph{(Filtration)} An exhaustive decreasing filtration $(\mathrm{Fil}^{i}D)_{i \in \ZZ}$ on $D$ by $(F_{\pri} \otimes_{\QQ_p} L)$-submodules;
\item \emph{(Frobenius)} A $\sigma_{\mathrm{Fr}}$ semi-linear isomorphism $\varphi : D \rightarrow D$ (where $\sigma_{\mathrm{Fr}}$ acts trivially on $L$);
\item \emph{(Monodromy)} A $(F_{\pri} \otimes_{\QQ_p} L)$-linear nilpotent operator $N:D\rightarrow D$ such that $N\varphi = p\varphi N$.
\end{itemize}
\end{definition}
Since we take $L/\QQ_{p}$ to be large enough to contain the image of all possible embeddings $\sigma:F_{\pri} \hookrightarrow \overline{\QQ}_{p}$, we have
\begin{equation} \label{eqn:decomposition1}
 F_{\pri} \otimes_{\QQ_{p}} L \cong \prod\limits_{\sigma:F_{\pri}\hookrightarrow L} L,\qquad (a \otimes b) \mapsto (\sigma(a)b)_{\sigma:F_{\pri}\hookrightarrow L},
\end{equation}
which gives the decomposition
\begin{equation} \label{eqn:decomposition2}
 D \cong \prod\limits_{\sigma:F_{\pri}\hookrightarrow L}D_{\sigma},\qquad D_{\sigma} := D \otimes_{F_{\pri} \otimes L,\sigma} L,
 \end{equation}
where each $D_{\sigma}$ is an $L$-vector space of dimension $r$ and is equipped with an $L$-linear action of $\varphi^{d}$ (where $d = [F_{\pri}:\QQ_{p}])$ and $N.$ The operator $\varphi$ (on $D$) induces a bijection $D_{\sigma} \sim D_{\sigma\circ\varphi^{-1}}$. In particular, when $p$ splits in $F$ there is just the identity embedding $\sigma_{\mathrm{Id}}:F_{\pri} \hookrightarrow L$. When $p$ is inert in $F$, we have two embeddings viz. the identity $\sigma_{\mathrm{Id}}$ and the lift of the Frobenius $\sigma_{\mathrm{Fr}}$. In this case, $\varphi$ induces a bijection
\begin{equation} \label{eqn:decompositionisomorphism}
D_{\sigma_{\mathrm{Id}}} \sim D_{\sigma_{\mathrm{Fr}}}.
\end{equation} 


\subsection{Monodromy modules from $\pri$-adic integration} \label{subsec:monodromyfrobenius}
We first treat the case when $p$ is inert in $F$. The $p$ split case is similar. We set
\[ \mathbf{D}_{f,L} := \mathbf{MS}_{\Gamma}(L)_{(f)}^{\vee} \oplus \mathbf{MS}_{\Gamma}(L)_{(f)}^{\vee}, \]
which is a two dimensional $L$-vector space (See Remark~\ref{rem:multiplicityone}), and let 
\begin{equation} \label{eqn:monodromyeqn1}
\mathbf{D}_{f} := \mathbf{D}_{f,\sigma_{\mathrm{Id}}} \times \mathbf{D}_{f,\sigma_{\mathrm{Fr}}}
\end{equation}   
where $\mathbf{D}_{f,\sigma_{\mathrm{Id}}}$ and $\mathbf{D}_{f,\sigma_{\mathrm{Fr}}}$ are both $\mathbf{D}_{f,L}$ as $L$-vector spaces but with the scalar action of $F_{\p}$ (as a subfield of $L$) given by the $\sigma_{\mathrm{Id}}$-embedding and $\sigma_{\mathrm{Fr}}$-embedding respectively, giving $\mathbf{D}_{f}$ the structure of a rank two $(F_{\pri} \otimes_{\QQ_{p}} L)$-module. We also equip $\mathbf{D}_{f}$ with:
\begin{itemize}
\item A \emph{Filtration} on $\mathbf{D}_{f}$ by $(F_{\pri} \otimes_{\QQ_{p}} L)$-submodules given by
\[ \mathbf{D}_{f} = \mathrm{Fil}^{0} \supsetneq \mathrm{Fil}^{1} = \ldots = \mathrm{Fil}^{k+1} \supsetneq \mathrm{Fil}^{k+2} = 0, \]
where
\[
\mathrm{Fil}^{\frac{k+2}{2}}\mathbf{D}_{f} := \{ [(-\mathcal{L}^{\sigma_{\mathrm{Id}}}_{\pri}x_{\sigma_{\mathrm{Id}}},x_{\sigma_{\mathrm{Id}}}),(-\mathcal{L}^{\sigma_{\mathrm{Fr}}}_{\pri}x_{\sigma_{\mathrm{Fr}}},x_{\sigma_{\mathrm{Fr}}})] : (-\cL_{\pri}^{\sigma}x_{\sigma},x_{\sigma}) \in \mathbf{D}_{f,\sigma}\}.
\]
\item A \emph{Frobenius} $\varphi$ on $\mathbf{D}_{f}$ defined as follows; writing $\mathbf{D}_{f}$ as 
\[ \mathbf{D}_{f} = \{(x_{\sigma_{\mathrm{Id}}},y_{\sigma_{\mathrm{Id}}},x_{\sigma_{\mathrm{Fr}}},y_{\sigma_{\mathrm{Fr}}})|(x_{\sigma},y_{\sigma}) \in \mathbf{D}_{f,\sigma}\}, \]
we define $\varphi : \mathbf{D}_{f} \rightarrow \mathbf{D}_{f}$ as
\[ \varphi : \big(x_{\sigma_{\mathrm{Id}}},y_{\sigma_{\mathrm{Id}}},x_{\sigma_{\mathrm{Fr}}},y_{\sigma_{\mathrm{Fr}}}\big) \longmapsto \big(U_{\pri}(x_{\sigma_{\mathrm{Fr}}}),pU_{\pri}(y_{\sigma_{\mathrm{Fr}}}),U_{\pri}(x_{\sigma_{\mathrm{Id}}}),pU_{\pri}(y_{\sigma_{\mathrm{Id}}})\big). \]
Note that $\varphi$ interchanges $\mathbf{D}_{f,\sigma_{\mathrm{Id}}}$ and $\mathbf{D}_{f,\sigma_{\mathrm{Fr}}}$. 
\item A \emph{Monodromy} $N$ on $\mathbf{D}_{f}$ defined by $N(x, y) = (y, 0)$. 
\end{itemize}

This gives $\mathbf{D}_{f}$ the structure of a rank two filtered $(\varphi, N)$-module over $F_{\pri}$ with coefficients in $L$, which we also write as $\mathbf{D}_{f} \in \textup{MF}(\varphi, N, F_{\pri}, L)$. Further, $\mathbf{D}_{f,L}$ (viewed as $\mathbf{D}_{f,\sigma}$) attains the structure of a filtered $L$-vector space of dimension two, with the filtration given by 
\begin{equation} \label{eqn:filtartiononsigmacomponents}
\mathrm{Fil}^{j}(\mathbf{D}_{f,L}) = \mathrm{Fil}^{j}(\mathbf{D}_{f,\sigma}) \defeq \mathrm{Fil}^{j}(\mathbf{D}_{f}) \otimes_{F_{\pri} \otimes L, \sigma} L.
\end{equation} 
for $j \in [0, k+2]$.

When $p$ is split in $F$, we have only the identity $\sigma_{\mathrm{Id}}:F_{\pri} \hookrightarrow L$ to consider and as before, we set 
\[ \mathbf{D}_{f} := \mathbf{D}_{f,\sigma_{\mathrm{Id}}} = \mathbf{MS}_{\Gamma}(L)_{(f)}^{\vee} \oplus \mathbf{MS}_{\Gamma}(L)_{(f)}^{\vee} \] 
As above, $\mathbf{D}_{f}$ attains the structure of a filtered $(\varphi, N)$-module over $F_{\pri}\cong \QQ_{p}$ with coefficients in $L$.


\subsection{Fontaine--Mazur theory} \label{subsec:fontainemazur} Let $G_{F} := \mathrm{Gal}(\overline{\QQ}/F)$, and let 
\[ V_{p}(f) : G_{F} \rightarrow \mathrm{GL}_{2}(L) \]
be the two dimensional $p$-adic Galois representation associated to $f$ by the work of Harris--Soudry--Taylor, Taylor and Berger--Harcos (cf. \cite{Harris1993}, \cite{Taylor1994} and \cite{BH07}). We may assume that $V_{p}(f)$ takes values in our base field $L$. Since the Bianchi eigenform $f$ is new at $\pri$ and $\pri$ exactly divides the level $\cN$, we know that the cuspidal automorphic representation $\pi$ of $\GL_{2}(\mathbb{A}_{F})$ that $f$ generates is Steinberg at $\pri$. Conjecturally, the local Galois representation 
\[V_{p}(f)|_{G_{F_{\pri}}} : \mathrm{Gal}(\overline{\QQ}_{p}/F_{\pri}) \rightarrow \GL_{2}(\overline{\QQ}_p)\]
is semi-stable non-crystalline in the sense of $p$-adic Hodge theory. Assuming this, Fontaine's semi-stable Dieudonn\'{e} module  
\[ 
\mathbb{D}_{f} := \mathbb{D}_{\mathrm{st}}(V_{p}(f)) := (B_{\mathrm{st},F_{\pri}} \otimes_{\QQ_{p}} V_{p}(f)|_{G_{F_{\pri}}})^{\mathrm{Gal}(\overline{\QQ}_{p}/F_{\pri})} 
\]
is then an \textit{admissible} rank two filtered $(\varphi, N)$-module over $(F_{\pri} \otimes_{\QQ_{p}} L)$. We choose a pair of $\varphi$-eigenvectors $u_{1}, u_{2} \in \mathbb{D}_{f}$ such that $N(u_{1}) = u_{2}$. Then $\mathbb{D}_{f} \cong (F_{\pri} \otimes_{\QQ_{p}} L)u_{1} \oplus (F_{\pri} \otimes_{\QQ_{p}} L)u_{2}$ as an $(F_{\pri} \otimes_{\QQ_{p}} L)$-module, with filtration given by
\[\mathbb{D}_{f} = \mathrm{Fil}^{0}\supsetneq \mathrm{Fil}^{1} =\ldots= \mathrm{Fil}^{k+1}\supsetneq \mathrm{Fil}^{k+2}=0  \]
such that $\{\mathrm{Fil}^{i}\}_{1}^{k+1}$ is a rank one $(F_{\pri} \otimes_{\QQ_{p}} L)$-submodule of $\mathbb{D}_{f}$. The vector valued \textit{Fontaine--Mazur} $\mathcal{L}$-invariant, denoted $\mathcal{L}_{\mathrm{FM}} = \lbrace \mathcal{L}_{\mathrm{FM}}^{\sigma} \rbrace_{\sigma : F_{\pri}\hookrightarrow L}$, is defined to be the unique element such that
\[ \mathrm{Fil}^{\frac{k+2}{2}}(\mathbb{D}_{f}) = F_{\pri} \otimes_{\QQ_{p}} L\cdot(u_{1} - \mathcal{L}_{\mathrm{FM}}\cdot u_{2}).\] 

For each embedding $\sigma:F_{\pri} \hookrightarrow L$, $\mathbb{D}_{f,\sigma} \defeq \mathbb{D}_{f} \otimes_{F_{\pri}\otimes L, \sigma} L$ is a filtered $L$-vector space of dimension two with the filtration given by 
\begin{equation} \label{eqn:filtrationfontainemazur}
\mathrm{Fil}^{j}(\mathbb{D}_{f,\sigma}) = \mathrm{Fil}^{j}(\mathbb{D}_{f}) \otimes_{F_{\pri} \otimes L, \sigma} L
\end{equation}
for $j \in [0, k+2].$
\begin{conjecture}[Trivial zero conjecture] \label{conj:trivialzeroconjecture} 
For each $\sigma:F_{\pri}\hookrightarrow L$, the cohomological $\mathcal{L}$-invariant $\mathcal{L}^{\sigma}_{\pri}$ is the Fontaine--Mazur $\mathcal{L}$-invariant $\mathcal{L}_{\mathrm{FM}}^{\sigma}$ attached to $V_{p}(f)$.
\end{conjecture}
\begin{remark}
This is related to Conjectures 11.1 and 11.2 of \cite{BW19}. Analogous results are known for classical modular forms (initially \cite{GS93}, but intensively studied since; see the introduction to \cite{BDI10} or \cite[\S 4]{Colmez05} for a comprehensive survey of the literature) and Hilbert modular forms (see e.g.\ \cite{Spi14},\cite{Mok09} and \cite{BDJ17}). 
\end{remark}

\begin{lemma} \label{lem:basechangetrivialzero}
Suppose that $f$ is the base-change to $F$ of a classical modular form $\widetilde{f}$. Then Conjecture~\ref{conj:trivialzeroconjecture} is true.
\end{lemma}
\begin{proof}
Let $\cL_{p}(\widetilde{f})$ be the Darmon--Orton analytic $\cL$-invariant associated to the classical form $\widetilde{f}$. Note we have an equality $\cL_{p}(\widetilde{f}) = \cL_{\mathrm{FM}}(\widetilde{f})$ (see e.g. \ \cite[\S 4]{Colmez05}). When $p$ is inert in $F$, we have the following equality of cohomological $\cL$-invariants
\[ \cL_{\pri}^{\sigma_{\mathrm{Id}}} = \cL_{\pri}^{\sigma_{\mathrm{Fr}}} = \cL_{p}(\widetilde{f})\]
where the first equality follows by the Galois invariance of \cite[Lemma 3.1]{Lennart2019} and the second equality follows since $\mathcal{L}_\pri^{\sigma_{\mathrm{Id}}} + \cL_\pri^{\sigma_{\mathrm{Fr}}} = 2\cL_p(\widetilde{f})$ by Corollary \ref{cor:L-invariant} and \cite[Prop.~10.2]{BW19}. Similarly, we have an analogous equality of the Fontaine--Mazur $\cL$-invariants
\[ \cL_{\mathrm{FM}}^{\sigma_{\mathrm{Id}}} = \cL_{\mathrm{FM}}^{\sigma_{\mathrm{Fr}}} = \cL_{\mathrm{FM}}(\widetilde{f}) \]
since the Fontaine--Mazur $\cL$-invariant is stable under base--change (see e.g.\ \cite[\S 4.1]{RS12}). Putting this together gives the required equality for $p$ inert. The split case is handled similarly.
\end{proof}
Note that the jumps in filtration for both the monodromy modules $\mathbf{D}_{f}$ and $\mathbb{D}_{f}$ occur at $(0,k+1)$. Further, the Frobenius operators $\varphi_{\mathbf{D}_{f}}$ and $\varphi_{\mathbb{D}_{f}}$ also coincide. In particular, if we assume Conjecture~\ref{conj:trivialzeroconjecture} (in addition to the semi-stability of $\rho_f$) then we have:
\begin{theorem} \label{thm:identifcationofmonodromymodules}
There is an isomorphism
\[ \mathbf{D}_{f} \cong \mathbb{D}_{f}, \]
of filtered $(\varphi, N)$-modules over $F_{\pri}$, with coefficients in $L$, which is stable under base--change.
\end{theorem}
\begin{proof} See \cite[Prop.~ 4.6]{RS12}, which treats the generic case.
\end{proof}
\begin{remark} \label{rem:choiceofp-adiclog}
The definition of both the monodromy modules $\mathbf{D}_{f}$ and $\mathbb{D}_{f}$ depends on the choice of a branch of the $p$-adic logarithm $\mathrm{log}_{p} : \mathbb{C}_{p}^{\times} \rightarrow \mathbb{C}_{p}$. We shall assume that the same choice of the branch has been made for both the monodromy modules, and that for this branch $\mathrm{log}_{p}(p) = 0$.
\end{remark}
Note that the isomorphism of Theorem~\ref{thm:identifcationofmonodromymodules} induces an identification of the tangent spaces
\begin{equation} \label{eqn:isomorphismoftangentspaces}
\frac{\mathbf{D}_{f}}{\mathrm{Fil}^{\frac{k+2}{2}}(\mathbf{D}_{f})} \cong \frac{\mathbb{D}_{f}}{\mathrm{Fil}^{\frac{k+2}{2}}(\mathbb{D}_{f})}.
\end{equation} 
Let $\sigma:F_{\pri}\hookrightarrow L$ be any embedding. Then $\mathbf{D}_{f,\sigma} := \mathbf{D}_{f} \otimes_{F_{\pri} \otimes L,\sigma} L \cong \mathbf{D}_{f,L}$ and $\mathbb{D}_{f,\sigma} := \mathbb{D}_{f} \otimes_{F_{\pri} \otimes L,\sigma} L$ (which we henceforth will simply denote as $\mathbb{D}_{f,L}$) are both two dimensional vector spaces over $L$ with filtrations given by (\ref{eqn:filtartiononsigmacomponents}) and (\ref{eqn:filtrationfontainemazur}) respectively. On tensoring both sides of (\ref{eqn:isomorphismoftangentspaces}) over $F_{\pri}\otimes_{\QQ_{p}} L$ with $L$ via the projection induced by the embedding $\sigma$, we obtain an identification
\begin{equation} \label{eqn:identifcationofmonodromymodules}
\frac{\mathbf{D}_{f,L}}{\mathrm{Fil}^{\frac{k+2}{2}}(\mathbf{D}_{f,L})} \cong \frac{\mathbb{D}_{f,L}}{\mathrm{Fil}^{\frac{k+2}{2}}(\mathbb{D}_{f,L})}.
\end{equation}


\section{$\pri$-adic Abel--Jacobi maps in the Stark--Heegner setting}
For the rest of the paper, we fix an embedding $\sigma:F_{\pri} \hookrightarrow \overline{\QQ}_{p}$ and define a map
\[ 
\Phi_{f} = \Phi_{f,\sigma} \defeq -\Phi^{\mathrm{log},\sigma}_{f} \oplus \Phi^{\mathrm{ord}}_{f} : (\Delta_0 \otimes \textup{Div}^{0}(\uhp_{\pri}^{\textup{ur}}) \otimes V_{k,k})_{\Gamma} \rightarrow \mathbf{D}_{f,L}.
 \]
By Corollary~\ref{cor:L-invariant}, we know that
\begin{align*}
\mathrm{Fil}^{\frac{k+2}{2}}(\mathbf{D}_{f,L}) &= \{ (-\mathcal{L}^{\sigma}_{\pri}x,x) : x \in \mathbf{MS}_{\Gamma}(L)_{(f)}^{\vee}\}\\
& = \mathrm{Im}(\Phi_{f}\circ\delta),
 \end{align*}
where $\delta :\tupH_{1}(\Gamma,\Delta_0 \otimes V_{k,k}) \rightarrow  (\Delta_0 \otimes \textup{Div}^{0}(\uhp_{\pri}^{\textup{ur}}) \otimes V_{k,k})_{\Gamma}$ is the connecting morphism of (\ref{eqn:homologyconnecting}). Recall the $\Gamma$-homology exact sequence
\begin{multline} \label{eqn:homologylongexactseq1}
\cdots \rightarrow \tupH_{i}(\Gamma, \Delta_0 \otimes \textup{Div}^{0}(\uhp_{\pri}^{\textup{ur}}) \otimes V_{k,k})  \rightarrow \tupH_{i}(\Gamma, \Delta_0 \otimes \textup{Div}(\uhp_{\pri}^{\textup{ur}}) \otimes V_{k,k}) \\
\rightarrow \tupH_{i}(\Gamma,\Delta_0 \otimes V_{k,k}) \rightarrow \cdots
\end{multline}
obtained from the short exact sequence (\ref{eqn:divisorexactsequence}). In particular, we have
\begin{equation} \label{eqn:homologylongexactseq2}
 \frac{(\Delta_0 \otimes \textup{Div}^{0}(\uhp_{\pri}^{\textup{ur}}) \otimes V_{k,k})_{\Gamma}}{\delta(\tupH_{1}(\Gamma,\Delta_0 \otimes V_{k,k}))} \xhookrightarrow{\partial_{1}} (\Delta_0 \otimes \textup{Div}(\uhp_{\pri}^{\textup{ur}}) \otimes V_{k,k})_{\Gamma} \xrightarrow{\partial_{2}} (\Delta_0 \otimes V_{k,k})_{\Gamma}.
 \end{equation}

\begin{definition} \label{defn:p-adicAbeliJacobi}
A \emph{$\pri$-adic Abel--Jacobi map} is a morphism
\[ \Phi^{\mathrm{AJ}} : (\Delta_0 \otimes \textup{Div}(\uhp_{\pri}^{\textup{ur}}) \otimes V_{k,k})_{\Gamma} \rightarrow \mathbf{D}_{f,L}/\mathrm{Fil}^{\frac{k+2}{2}}(\mathbf{D}_{f,L}) \]
such that the following diagram is commutative:
\begin{equation}
\xymatrix{
 \frac{(\Delta_0 \otimes \textup{Div}^{0}(\uhp_{\pri}^{\textup{ur}}) \otimes V_{k,k})_{\Gamma}}{\delta(\tupH_{1}(\Gamma,\Delta_0 \otimes V_{k,k}))} \ar@{^{(}->}[d]^{\partial_1} \ar[r]^<<<<<{\Phi_{f}} &     \mathbf{D}_{f,L}/\mathrm{Fil}^{\frac{k+2}{2}}(\mathbf{D}_{f,L}) \\
(\Delta_0 \otimes \textup{Div}(\uhp_{\pri}^{\textup{ur}}) \otimes V_{k,k})_{\Gamma} \ar@{-->}[ur]^{\Phi^{\mathrm{AJ}}} 
}
\end{equation}
In other words, a $\pri$-adic Abel--Jacobi map $\Phi^{\mathrm{AJ}}$ is a lift of the morphism $\Phi_{f}$.
\end{definition}
\begin{remark} \label{rem:abeljacobi}
Such commutative diagrams always exist, but they might not be unique. The chief obstruction comes from the fact that $(\Delta_0 \otimes V_{k,k})_{\Gamma}$ is non-trivial. In fact, let
\[ 
\Phi^{\mathrm{AJ}}_{i} : (\Delta_0 \otimes \textup{Div}(\uhp_{\pri}^{\textup{ur}}) \otimes V_{k,k})_{\Gamma} \rightarrow \mathbf{D}_{f,L}/\mathrm{Fil}^{\frac{k+2}{2}}(\mathbf{D}_{f,L}), 
\]
for $i \in \{1,2\}$, be any two $\pri$-adic Abel--Jacobi maps. Then, by Defn.~\ref{defn:p-adicAbeliJacobi}, we have
\[ 
\Phi^{\mathrm{AJ}}_{1}\big(\partial_1(-)\big) = \Phi_{f}(-) = \Phi^{\mathrm{AJ}}_{2}\big(\partial_1(-)\big). 
\]
In other words, 
\[ 
\mathrm{ker}(\Phi^{\mathrm{AJ}}_{1} - \Phi^{\mathrm{AJ}}_{2}) \supseteq \mathrm{Im}(\partial_1) = \mathrm{ker}(\partial_2),
\]
since $\partial_1$ and $\partial_2$ are connecting maps of the $\Gamma$-homology exact sequence (\ref{eqn:homologylongexactseq2}). In the next section, in Theorem~\ref{thm:padicabeljacobiimageofdarmoncycles}, we will show that the $\pri$-adic Abel--Jacobi image of our Stark--Heegner cycles is independent of the choice of a $\pri$-adic Abel--Jacobi map. 
\end{remark}

\section{Stark--Heegner cycles for Bianchi modular forms} \label{sec:stark-heegnercycles}
In this section, we define \emph{Stark--Heegner cycles}, which should be regarded as higher weight analogues of the Stark--Heegner points on elliptic curves over imaginary quadratic fields defined by Trifkovi\'c in \cite{Tri06}.

\subsection{Construction of homology classes}\label{sec:construction of homology classes}
As before, let $K/F$ be the quadratic extension of relative discriminant $\mathcal{D}_{K}$ prime to $\mathcal{N} = \pri\mathcal{M}$ that satisfies the Stark--Heegner hypothesis. In particular, the completion $K_{\pri}$ of $K$ at the prime $\pri$ is a quadratic unramified extension of $F_{\mathfrak{p}}$. Fix $\delta_{K} \in \mathcal{O}_{K}\backslash\mathcal{O}_{F}$ such that $\delta_{K}^{2} \in \mathcal{O}_{F}$ is a generator of the discriminant ideal (recall that $F$ has class number 1). We may regard $\delta_{K}$ as an element of $K_{\mathfrak{p}}$ via the fixed embedding $\iota_{p}$. Let
\[ 
	\mathcal{R} :=  R \cap M_2(F) = \left\{\smallmatrd{a}{b}{c}{d} \in M_{2}\left(\mathcal{O}_{F}\left[\tfrac{1}{\mathfrak{p}}\right]\right) \big\vert c \in \mathcal{M}\right\}. 
\]  
If we let $\mathcal{R}_{1}^{\times}$ denote the elements of $\mathcal{R}$ of determinant $1$, then $\Gamma$ is the image of $\cR_1^\times$ under projection to $\textup{PGL}_{2}(\mathcal{O}_{F}[1/\mathfrak{p}])$. Let $\mathcal{O}$ be an $\mathcal{O}_{F}[1/\mathfrak{p}]$-order of conductor $\mathcal{C}$ prime to $\mathcal{D}_{K}\mathcal{N}$. 
\begin{definition} \label{defn:optimalembedding}
We say an embedding $\Psi : K \hookrightarrow M_{2}(F)$ is \emph{optimal} if $\Psi(K) \cap \mathcal{R} = \Psi(\mathcal{O})$.   
\end{definition}
Denote the set of $\mathcal{O}_{F}[1/\mathfrak{p}]$-optimal embeddings by $\textup{Emb}(\mathcal{O}, \mathcal{R})$. To an embedding $\Psi \in \textup{Emb}(\mathcal{O}, \mathcal{R})$, we associate the following data.
\begin{itemize}
\item The two points $\tau_{\psi}$ and $\tau_{\Psi}^{\theta} \in \mathcal{H}_{\mathfrak{p}}^{\textup{ur}}(K) := \mathcal{H}_{\mathfrak{p}}^{\textup{ur}} \cap K$ that are fixed by the action of $\Psi(K^{\times})$. Here $\tau_{\Psi}^{\theta} = \theta(\tau_{\Psi})$ for $\theta \in \textup{Gal}(K/F)$ where $\theta \neq \textup{id}$.
\item The fixed vertex $v_{\Psi} \in \mathcal{V}$ in the Bruhat-Tits tree for the action of $\psi(K^{\times})$ on $\mathcal{V}$.
\item The polynomial $P_{\Psi}(X,\overline{X}) := (cX^2 + (a - d)X - b)(\overline{c}\overline{X}^2 + (\overline{a} - \overline{d})\overline{X} - \overline{b})\in V_{2,2}$,
where $\Psi(\delta_{K}) = \smallmatrd{a}{b}{c}{d}$. 
\item Let $u$ be a fixed generator of $\mathcal{O}_{1}^{\times}/\{\textup{torsion}\} \cong \ZZ$ (by Dirichlet's Unit theorem), where $\mathcal{O}_{1} := \{ x \in \mathcal{O}\mid N_{K/F}(x) = 1 \}$. Let $\gamma_{\Psi} := \Psi(u)$ and $\Gamma_{\Psi}$ be the cyclic subgroup of $\Gamma$ generated by $\gamma_{\Psi}$. In particular $\Gamma_{\Psi} = \textup{Stab}(\Psi) \subseteq \Gamma$ and $P_{\Psi} \in (V_{2,2})^{\Gamma_{\Psi}}$.  
\end{itemize}

\begin{remark}
Compare to the data used in \cite[\S7.1]{BW19}; there, objects analogous to those above -- for the \emph{split} embeddings $F\times F \hookrightarrow M_2(F)$ -- were defined directly, with no reference to optimal embeddings. The results \emph{op.\ cit}.\ demonstrate the close relationship between these invariants, the double integrals defined above, and special values of $L$-functions (and $p$-adic $L$-functions) of Bianchi modular forms. In particular, the invariants attached to split embeddings have already appeared in this paper, giving the classes $S_\chi$ used in the proof of Theorem~\ref{thm:H1H0}.
\end{remark}

We say that $\Psi$ has \textit{positive} (resp.\ \emph{negative}) \emph{orientation} if $v_{\Psi} \in \mathcal{V}^{+}$ (resp. $\mathcal{V}^{-}$). Then
\[ \textup{Emb}(\mathcal{O}, \mathcal{R}) = \textup{Emb}^{+}(\mathcal{O}, \mathcal{R}) \sqcup \textup{Emb}^{-}(\mathcal{O}, \mathcal{R}) \]
where $\textup{Emb}^{\pm}(\mathcal{O}, \mathcal{R})$ denotes the set of embeddings with positive/negative orientation. The group $\Gamma$ acts naturally on the set $ \textup{Emb}(\mathcal{O}, \mathcal{R})$ by conjugation, preserving the subsets $ \textup{Emb}^{\pm}(\mathcal{O}, \mathcal{R})$. In fact it can be shown that the association
\[ \Psi \mapsto (\tau_{\Psi}, P_{\psi}, \gamma_{\Psi}) \]
under conjugation by $\gamma \in \Gamma$ satisfies
\begin{equation} \label{eqn:conjugationaction}
(\tau_{\gamma\Psi\gamma^{-1}},P_{\gamma\Psi\gamma^{-1}},\gamma_{\gamma\Psi\gamma^{-1}}) = (\gamma\cdotspace\tau_{\Psi}, \gamma^{-1}\cdotspace P_{\Psi}, \gamma\gamma_{\Psi}\gamma^{-1}).
\end{equation}
Similarly, let $\Psi^{\theta} \in \textup{Emb}(\mathcal{O}, \mathcal{R})$ be the embedding given by $\Psi^{\theta}(-) := \Psi(\theta(-))$ for $\theta \in \textup{Gal}(K/F)$, $\theta \neq \textup{id}$. Then, we have 
\[ 
(\tau_{\Psi^{\theta}},P_{\Psi^{\theta}},\gamma_{\Psi^{\theta}}) = (\tau_{\Psi}^{\theta}, -P_{\Psi}, \gamma_{\Psi}^{-1}).
 \]
Once we fix a cusp $x \in \mathbb{P}^{1}(F)$, we define
\[ 
\mathrm{D} : \textup{Emb}(\mathcal{O},\mathcal{R}) \longrightarrow \Delta_{0} \otimes \textup{Div}(\mathcal{H}_{\mathfrak{p}}^{\textup{ur}}) \otimes V_{k,k},
 \]
\[ 
\mathrm{D}_{\Psi} := \mathrm{D}(\Psi) := (\gamma_{\Psi}\cdotspace x - x)\otimes \tau_{\Psi} \otimes \delta_{K}^{-k/2}P_{\Psi}^{k/2}. 
\]
\begin{lemma} \label{lem:darmoncycles1}
The image of $\mathrm{D}_{\Psi}$ in $(\Delta_{0} \otimes \textup{Div}(\mathcal{H}_{\mathfrak{p}}^{\textup{ur}}) \otimes V_{k,k})_{\Gamma}$, which we denote by $[\mathrm{D}_{\Psi}]$, remains the same if we replace $x$ with any $y \in {\Gamma}x$. Further, $[\mathrm{D}_\Psi]$ is invariant under the conjugation action of $\Gamma$ on $\textup{Emb}(\mathcal{O}, \mathcal{R})$. 
\end{lemma}
\begin{proof}
The first part of the lemma is easy to see. The second part follows from (\ref{eqn:conjugationaction}).
\end{proof}
In particular, we get a well defined map
\begin{equation} \label{eqn:darmoncycles2}
\mathrm{D} : \Gamma/\textup{Emb}(\mathcal{O},\mathcal{R}) \longrightarrow (\Delta_{0} \otimes \textup{Div}(\mathcal{H}_{\mathfrak{p}}^{\textup{ur}}) \otimes V_{k,k})_{\Gamma}.
\end{equation}
\begin{definition} \label{defn:stark-heegnercycle1}
We call $[\mathrm{D}_\Psi]$ the \emph{Stark--Heegner cycle} attached to the conjugacy class $[\Psi]$.
\end{definition}
\begin{theorem} \label{thm:padicabeljacobiimageofdarmoncycles}
The $\pri$-adic Abel--Jacobi image of the Stark--Heegner cycle $\mathrm{D}_{[\Psi]}$ is independent of the choice of a $\pri$-adic Abel--Jacobi map. That is, if  
\[ 
\Phi^{\mathrm{AJ}}_{i} : (\Delta_0 \otimes \textup{Div}(\uhp_{\pri}^{\textup{ur}}) \otimes V_{k,k})_{\Gamma} \longrightarrow \mathbf{D}_{f,L}/\mathrm{Fil}^{\frac{k+2}{2}}(\mathbf{D}_{f,L}) , \hspace{12pt} i = 1,2
\] 
 are any two $\pri$-adic Abel--Jacobi maps lifting $\Phi_{f}$ in the sense of Defn.~ \ref{defn:p-adicAbeliJacobi}, then 
\[ 
\Phi^{\mathrm{AJ}}_{1}\left(\left[(\gamma_{\Psi}\cdotspace x - x) \otimes \tau_{\Psi} \otimes \delta_{K}^{-k/2}P_{\Psi}^{k/2}\right]\right) = \Phi^{\mathrm{AJ}}_{2}\left(\left[(\gamma_{\Psi}\cdotspace x - x) \otimes \tau_{\Psi} \otimes \delta_{K}^{-k/2}P_{\Psi}^{k/2}\right]\right). 
\]
\end{theorem}
\begin{proof}
Note that
\begin{equation} \label{eqn:connecting}
\partial_{2}\left(\left[(\gamma_{\Psi}\cdotspace x - x) \otimes \tau_{\Psi} \otimes \delta_{K}^{-k/2}P_{\Psi}^{k/2}\right]\right) = \left[(\gamma_{\Psi}\cdotspace x - x) \otimes \delta_{K}^{-k/2}P_{\Psi}^{k/2}\right] \in (\Delta^{0} \otimes V_{k,k})_{\Gamma}. 
\end{equation}
Consider the exact sequence
\begin{equation} \label{eqn:divisorexactseq2}
0 \rightarrow \Delta^{0} \otimes V_{k,k} \rightarrow \Delta \otimes V_{k,k} \xrightarrow{\mathrm{deg}} V_{k,k} \rightarrow 0;
\end{equation}
on taking $\Gamma$-homology, we get the connecting morphisms 
\begin{equation} \label{eqn:homologyexactseq3}
\cdots \rightarrow \tupH_{1}(\Gamma,V_{k,k}) \xrightarrow{d^1} (\Delta^{0} \otimes V_{k,k})_{\Gamma} \xrightarrow{d^2} (\Delta \otimes V_{k,k})_{\Gamma} \rightarrow \cdots.
\end{equation}
We show that $[(\gamma_{\Psi}\cdotspace x - x) \otimes \delta_{K}^{-k/2}P_{\Psi}^{k/2}] \in \mathrm{ker}(d^{2})$. By construction, $P_{\Psi}^{k/2}$ is fixed by $\gamma_{\Psi}$. Hence
\[ 
 \left[(\gamma_{\Psi}\cdotspace x) \otimes \delta_{K}^{-k/2}P_{\Psi}^{k/2}\right] = \left[(x) \otimes \delta_{K}^{-k/2}P_{\Psi}^{k/2}\right] 
 \]
as elements in the $\Gamma$-coinvariants $(\Delta \otimes V_{k,k})_{\Gamma}$. In other words,
\[
 \left[(\gamma_{\Psi}\cdotspace x - x) \otimes \delta_{K}^{-k/2}P_{\Psi}^{k/2}\right] \in \mathrm{ker}(d^{2}) = \mathrm{Im}(d^{1}).
 \]
In Appendix~\ref{appendix} (see Corollary~\ref{cor:vanishingofH1}), we show that $\tupH_{1}(\Gamma,V_{k,k}) = 0$. This then forces
\[
\left[(\gamma_{\Psi}\cdotspace x - x) \otimes \delta_{K}^{-k/2}P_{\Psi}^{k/2}\right] = 0 \in (\Delta^{0} \otimes V_{k,k})_{\Gamma}.
 \]
By (\ref{eqn:connecting}), the Stark--Heegner cycle 
\[\mathrm{D}_{[\Psi]} = \left[(\gamma_{\Psi}\cdotspace x - x) \otimes \tau_{\Psi} \otimes \delta_{K}^{-k/2}P_{\Psi}^{k/2}\right] \in \mathrm{ker}(\partial_2)\]
where $\partial_{2} : (\Delta_0 \otimes \textup{Div}(\uhp_{\pri}^{\textup{ur}}) \otimes V_{k,k})_{\Gamma} \rightarrow (\Delta_0 \otimes V_{k,k})_{\Gamma}$ is the connecting morphism in (\ref{eqn:homologylongexactseq2}). The required equality then follows by Remark~\ref{rem:abeljacobi}.
\end{proof}
\begin{definition} \label{defn:starkheegnercohomologyclass}
The \emph{Stark--Heegner cohomology class} associated to a conjugacy class of embeddings $[\Psi] \in \Gamma/\mathrm{Emb}(\cO,\mathcal{R})$ is defined as 
\[
	\mathfrak{s}_{[\Psi]} := \Phi^{\mathrm{AJ}}(\mathrm{D}_{[\Psi]}) \in \mathbf{D}_{f,L}/\mathrm{Fil}^{\frac{k+2}{2}}(\mathbf{D}_{f,L}).
\]
\end{definition}
As in \cite[Prop.\ 5.8]{Dar01} and \cite[Prop.\ 2]{Tri06}, there is a natural action of $\textup{Pic}(\mathcal{O})$ on the set $\Gamma/\textup{Emb}(\mathcal{O},\mathcal{R})$. Let $H_{\mathcal{C}}$ be the ring class field of conductor $\mathcal{C}$. By the reciprocity isomorphism
\begin{equation} \label{eqn:cftisomorphism}
\mathrm{rec} : \textup{Pic}(\mathcal{O}) \cong \textup{Gal}(H_{\mathcal{C}}/K)
\end{equation}
of class field theory, we get a transported action of $\textup{Gal}(H_{\mathcal{C}}/K)$ on $\Gamma/\textup{Emb}(\mathcal{O},\mathcal{R})$. 
\begin{definition} \label{defn:stark-heegnercycle2}
Let $\chi : \textup{Gal}(H_{\mathcal{C}}/K) \rightarrow \mathbb{C}^{\times}$ be any character. The \emph{$\chi$-twisted Stark--Heegner cycle} is then defined as 
\[ 
\mathrm{D}_{\chi} := \sum\limits_{\tau \in \textup{Gal}(H_{\mathcal{C}}/K)} \chi^{-1}(\tau)\mathrm{D}_{[\textup{rec}^{-1}(\tau)\cdotspace \Psi]} \in (\Delta_{0} \otimes \textup{Div}(\mathcal{H}_{\mathfrak{p}}^{\textup{ur}}) \otimes V_{k,k})_{\Gamma} \otimes \chi,
 \]
where $(-) \otimes \chi$ denotes suitable scalar extension by $\chi$. We denote $\mathfrak{s}_{\chi} := \Phi^{\mathrm{AJ}}(\mathrm{D}_{\chi})$. 
\end{definition}

\subsection{Conjectures on the global rationality}
Let $V_{p}(f)$ be the continuous two--dimensional representation of $G_{F} :=\mathrm{Gal}(\overline{\QQ}/F)$ as before. Recall that we are assuming the conjectural semi-stability of the local Galois representation $V_{p}(f)|_{G_{F_{\pri}}}$ (in the sense of $p$-adic Hodge theory). Under this assumption, we have an isomorphism under the Bloch--Kato exponential \cite{BK07} as in \cite[Lemma 2.1]{IS03}
\begin{equation} \label{eqn:blochkatoexp}
\mathrm{exp}_{\mathrm{BK}} : \frac{\mathbb{D}_{f,L}}{\textup{Fil}^{\frac{k+2}{2}}(\mathbb{D}_{f,L})} \cong \textup{H}^{1}_{\textup{st}}\big(L, V_{p}(f)(\tfrac{k+2}{2})\big). 
\end{equation}
where we now regard $V_{p}(f)$ as a semi-stable representation of $G_{L} \defeq \mathrm{Gal}(\overline{\QQ}_{p}/L)$ by restricting to the action of $G_{L} \subset G_{F_{\pri}}$ (afforded by the inclusion $F_{\pri} \xhookrightarrow{\sigma} L$). If we also assume Conjecture~\ref{conj:trivialzeroconjecture}, Theorem~\ref{thm:identifcationofmonodromymodules} and (\ref{eqn:identifcationofmonodromymodules}) provides us with an isomorphism
\[ 
\frac{\mathbf{D}_{f,L}}{\textup{Fil}^{\frac{k+2}{2}}(\mathbf{D}_{f,L})} \cong \frac{\mathbb{D}_{f,L}}{\textup{Fil}^{\frac{k+2}{2}}(\mathbb{D}_{f,L})}, 
\]
and hence by means of the Bloch--Kato exponential (\ref{eqn:blochkatoexp}), we may regard the Stark--Heegner cohomology classes as elements
\begin{equation} \label{eqn:localstarkheegener}
\mathfrak{s}_{\Psi} \in \textup{H}^{1}_{\textup{st}}(L, V_{p}(f)(\tfrac{k+2}{2}))\mbox{      and       }\mathfrak{s}_{\chi} \in \textup{H}^{1}_{\textup{st}}(L(\chi), V_{p}(f)(\tfrac{k+2}{2})), 
\end{equation}
where $L(\chi)$ is the field generated over $L$ by the values of $\chi$. Recall our running assumption that $L \supseteq K_{\pri}$. Since the prime $\pri$ splits completely in $H_{\mathcal{C}}$, we have an embedding $H_{\mathcal{C}} \xhookrightarrow{\iota_{p}} K_{\pri} \subseteq L$ that induces the restriction map 
\begin{equation} \label{eqn:restriction2}
\textup{res}_{\mathfrak{p}} : \textup{Sel}_{\sst}(H_{\mathcal{C}}, V_{p}(f)(\frac{k+2}{2})) \rightarrow \tupH^{1}_{\sst}(L, V_{p}(f)(\frac{k+2}{2})) \cong \frac{\mathbb{D}_{f,L}}{\textup{Fil}^{\frac{k+2}{2}}(\mathbb{D}_{f,L})} \cong \frac{\mathbf{D}_{f,L}}{\textup{Fil}^{\frac{k+2}{2}}(\mathbf{D}_{f,L})}.
\end{equation}

The following is our main conjecture about the global rationality of these Stark--Heegner cycles.
\begin{conjecture} \label{conj:mainconjecture}
\begin{itemize} \item[(i)] (Global rationality) For any optimal embedding $\Psi \in \mathrm{Emb}(\cO,\mathcal{R})$ of conductor $\mathcal{C}$ relatively prime to $\mathcal{ND}_{K}$, there exists a global Selmer class $\mathcal{S}_{\Psi} \in \textup{Sel}_{\sst}(H_{\mathcal{C}}, V_{p}(\frac{k+2}{2}))$ such that
\[ 
\mathfrak{s}_{\Psi} = \textup{res}_{\mathfrak{p}}(\mathcal{S}_{\Psi}) \in \tupH^{1}_{\sst}\bigg(L, V_{p}(\tfrac{k+2}{2})\bigg).
 \]
\item[(ii)](Shimura reciprocity) For any $\Psi \in \mathrm{Emb}(\cO,\mathcal{R})$ and $\mathfrak{a} \in \mathrm{Pic}(\cO)$,
\[ \mathrm{res}_{\pri}(\mathcal{S}_{\Psi}^{\tau}) = \mathfrak{s}_{\mathfrak{a}\cdotspace \Psi} \]
for $\tau = \mathrm{rec}(\mathfrak{a})$.
\item[(iii)] (Twisted rationality) Let $\chi : \mathrm{Gal}(H_{\mathcal{C}}/K) \rightarrow \mathbb{C}^{\times}$ be a character. Then there exists $\mathcal{S}_{\chi} \in \textup{Sel}_{\sst}(H_{\chi}, V_{p}(\frac{k+2}{2}))^{\chi}$ such that
\[ \mathrm{res}_{\pri}(\mathcal{S}_{\chi}) = \mathfrak{s}_{\chi} \]
where $H_{\chi}$ is the abelian sub-extension of $H_{\mathcal{C}}$ cut out by $\chi$ and $(-)^{\chi}$ denotes the $\chi$-isotypical subspace.
\end{itemize}
\end{conjecture}
Finally, we also formulate a Gross--Zagier type criterion for non-vanishing that relates the base-change $L$-function to our Stark--Heegner cycles.
\begin{conjecture} \label{conj:grosszagier}
Let $\chi : \mathrm{Gal}(H_{\mathcal{C}}/K) \rightarrow \mathbb{C}^{\times}$ be a character. Then
\[ \mathfrak{s}_{\chi} \neq 0 \hspace{12pt}\text{implies}\hspace{12pt} L'(f/K,\chi,\tfrac{k+2}{2}) := \frac{d}{ds}L(f/K, \chi, s)\big|_{s = \frac{k+2}{2}} \neq 0. \]
\end{conjecture}

\begin{remark} \label{rem:choiceofisomorphism}
The image of the global Selmer group $\textup{Sel}_{\sst}(H_{\mathcal{C}}, V_{p}(f)(\frac{k+2}{2}))$ under (\ref{eqn:restriction2}) is a Hecke submodule of $\mathbf{D}_{f,L}/\textup{Fil}^{\frac{k+2}{2}}(\mathbf{D}_{f,L})$. It can be shown that every automorphism of $\mathbf{D}_{f,L}$ acts on the quotient space $\mathbf{D}_{f,L}/\textup{Fil}^{\frac{k+2}{2}}(\mathbf{D}_{f,L})$ via multiplication by an element in an appropriate Hecke algebra $\mathbb{T}_{\Gamma}$ (c.f. \cite[Lemma 4.4]{RS12}). Whilst the isomorphism of Theorem~\ref{thm:identifcationofmonodromymodules} depended on the choice of the branch of the $p$-adic logarithm, the image of the global Selmer group in $\mathbf{D}_{f,L}/\textup{Fil}^{\frac{k+2}{2}}(\mathbf{D}_{f,L})$ is independent of this choice.
\end{remark} 
\subsection{Concluding Remarks} 
\begin{itemize}
\item[(1)] It is possible to formulate Conjecture~\ref{conj:grosszagier} independent of the truth of Conjecture~\ref{conj:mainconjecture}, but a proof of either seems to be a long way off; even their counterparts for classical modular forms (cf. \cite{RS12} and \cite{Dar01}) are widely open except in a few cases. 
\item[(2)] A particular case where Conjecture~\ref{conj:mainconjecture} holds would be when the Bianchi eigenform $f$ is the base-change of a classical eigenform $\widetilde{f}$ of weight $k+2$ and level $N$ (such that $N\cO_{F} = \mathcal{N}$) which is new at the prime $p$, and the character $\chi : \mathrm{Gal}(H_{K}/K) \rightarrow \mathbb{C}^{\times}$ is an unramified quadratic (genus) character. Note that in this case, Conjecture~\ref{conj:trivialzeroconjecture} is known to hold by Lemma~\ref{lem:basechangetrivialzero}. In particular, one can show that the $\pri$-adic Abel--Jacobi image of the Stark--Heegner cycles considered here coincides with that of classical Heegner cycles considered in \cite{SevHeeg}, similar to the strategy pursued, for example, in \cite{10.2307/40345466} and \cite{Sev1}. This is the subject of an ongoing project of the first author.
\end{itemize}
\appendix
\section{Cohomology of $\Gamma$} \label{appendix}
The goal of this appendix is to show that $\tupH_{1}(\Gamma,V_{k,k})$ vanishes, as used in Theorem~\ref{thm:padicabeljacobiimageofdarmoncycles}. 
\subsection{Cuspidal and Eisenstein Cohomology}
Let $\cH_{3}$ be the hyperbolic $3$-space, fix a level group $U_{0} \subseteq \GL_{2}(\widehat{\cO_{F}})$ and set
\[ 
\Gamma_{0} \defeq \mathrm{SL}_{2}(F) \cap [\GL_{2}(\mathbb{C})U_{0}] \ \subseteq \mathrm{SL}_{2}(\cO_{F}). 
\]
The Borel--Serre compactification of $\Gamma_{0}\backslash\cH_{3}$ is $\Gamma_{0}\backslash\overline{\cH}_{3}$, where
\[ \overline{\cH}_{3} := \cH_{3} \sqcup \bigsqcup_{x \in \mathbb{P}^{1}(F)} \R^2.\]

We let $\mathcal{V}_{k,k}^\vee(\mathbb{C}_{p})$ denote the local system on $\Gamma_{0}\backslash\cH_{3}$ associated to $V_{k,k}^\vee(\mathbb{C}_{p}) = \mathrm{Hom}(V_{k,k}(\C_p),\C_p).$ By an abuse of notation, we also use $\mathcal{V}_{k,k}^\vee(\mathbb{C}_{p})$ to denote the associated local system on the compactification $\Gamma_{0}\backslash\overline{\cH}_{3}$ and on the boundary $\partial(\Gamma_{0}\backslash\overline{\cH}_{3})$. The cohomology of these spaces is related by the excision long exact sequence
	\begin{equation}\label{eqn:boundaryrestriction}
		\cdots \to \h^i_{\mathrm{c}}(\Gamma_0\backslash \cH_3, \cV_{k,k}^\vee(\Cp)) \to \h^{i}(\Gamma_0\backslash\cH_3,\cV_{k,k}^\vee(\Cp))  \xrightarrow{\ \mathrm{res}_\partial \ } \h^{i}(\partial(\Gamma_0\backslash\overline{\cH}_3),\cV_{k,k}^\vee(\Cp)) \to \cdots.
	\end{equation}

We recall from \cite[\S IV]{Har87} that $\mathrm{res}_{\partial}$ admits a section $i$ on its image; we define the \emph{Eisenstein subspace}
	\[
		\h^i_{\mathrm{Eis}}(\Gamma_0\backslash\cH_3, \cV_{k,k}^\vee(\Cp)) := \mathrm{Im}(i\circ\mathrm{res}_\partial) \subset \h^i(\Gamma_0\backslash\cH_3, \cV_{k,k}^\vee(\Cp)).
	\]
We then have a Hecke-stable decomposition into cuspidal and Eisenstein cohomology
\begin{equation}\label{eq:cusp + eis}
\tupH^{i}(\Gamma_{0}\backslash\cH_{3}, \mathcal{V}_{k,k}^\vee(\mathbb{C}_{p})) = \tupH^{i}_{\mathrm{cusp}}(\Gamma_{0}\backslash\cH_{3}, \mathcal{V}_{k,k}^\vee(\mathbb{C}_{p})) \oplus \tupH^{i}_{\mathrm{Eis}}(\Gamma_{0}\backslash\cH_{3}, \mathcal{V}_{k,k}^\vee(\mathbb{C}_{p})).
\end{equation}
We can describe $\h^i_{\mathrm{cusp}}$ using the Eichler--Shimura--Harder isomorphism (again, see \cite{Har87}):
\begin{theorem}[Harder] \label{thm:eichlershimuraharder}
\begin{itemize}
\item[(i)] $\tupH^{i}_{\mathrm{cusp}}(\Gamma_{0}\backslash\cH_{3}, \mathcal{V}_{k,k}^\vee(\mathbb{C}_{p})) = 0$ unless $i = 1,2$.
\item[(ii)] $\tupH^{1}_{\mathrm{cusp}}(\Gamma_{0}\backslash\cH_{3}, \mathcal{V}_{k,k}^\vee(\mathbb{C}_{p})) \cong \tupH^{2}_{\mathrm{cusp}}(\Gamma_{0}\backslash\cH_{3}, \mathcal{V}_{k,k}^\vee(\mathbb{C}_{p})) \cong S_{k,k}(U_{0}, \mathbb{C}_{p})$.
\item[(iii)] $\tupH^{0}_{\mathrm{Eis}}(\Gamma_{0}\backslash\cH_{3}, \mathcal{V}_{k,k}^\vee(\mathbb{C}_{p})) = 0$ unless $k = 0$, in which case it is $\mathbb{C}_{p}$.
\end{itemize}
\end{theorem}
Since $\Gamma_0$ is the fundamental group of $\Gamma_0\backslash\cH_3$, we have the standard identification of singular and group cohomology
\begin{equation}\label{eq:singular group 1}
	\h^i(\Gamma_0\backslash \cH_3, \cV_{k,k}^\vee(\C_p)) \cong \h^i(\Gamma_0, V_{k,k}^\vee(\C_p)),
\end{equation}
and via \eqref{eq:cusp + eis}, $\h^i(\Gamma_0, V_{k,k}^\vee(\C_p))$ decomposes into cuspidal and Eisenstein parts. We can similarly describe the boundary cohomology in group-theoretic language; let $\{s_{1},\ldots,s_{t}\}$ be a set of representatives of the cusps of $\Gamma_{0}$. We set
\[ \tupH^{i}_{\partial}(\Gamma_{0}, V_{k,k}^{\vee}) := \bigoplus\limits_{i = 1}^{t} \tupH^{i}(\Gamma_{0,s_{i}}, V_{k,k}^{\vee}), \]
where $\Gamma_{0,s_{i}}$ is the stabiliser of the cusp $s_{i}$ in $\Gamma_{0}$. Then we have an isomorphism
	\begin{equation}\label{eq:singular group 2}
		\h^i_\partial(\Gamma_0, V_{k,k}^\vee) \cong \h^i(\partial(\Gamma_0\backslash\overline{\cH}_3,\cV_{k,k}^\vee),
	\end{equation}
	which fits into a commutative diagram
\begin{equation} \label{eqn:identifyingcohomology}
\xymatrix@C=15mm{
\tupH^{i}(\Gamma_{0}, V_{k,k}^{\vee}) \ar[d]^{\cong}_{\eqref{eq:singular group 1}} \ar[r]^{\oplus_{i = 1}^{t}\mathrm{res}} & \tupH^{i}_{\partial}(\Gamma_{0}, V_{k,k}^{\vee}) \ar[d]^{\cong}_{\eqref{eq:singular group 2}}\\
\tupH^{i}(\Gamma_{0}\backslash\cH_{3}, \mathcal{V}_{k,k}^{\vee})  \ar[r]^-{\mathrm{res}_{\partial}} & \tupH^{i}(\partial(\Gamma_{0}\backslash\overline{\cH}_{3}), \mathcal{V}_{k,k}^{\vee}),}
\end{equation}
where the horizontal arrows are the restriction maps.
\subsection{$\pri$-new and $\pri$-old} Let $\alpha \in \widetilde{\Gamma}/\Gamma$ be any element normalising $\Gamma_{0}(\cN)$ and set $\Gamma_{0}'(\cM) := \alpha\Gamma_{0}(\cM)\alpha^{-1}$. We have the obvious restriction homomorphisms
\begin{align*} \mathrm{Res} : \tupH^{1}(\Gamma_{0}(\cM), V_{k,k}^{\vee}) &\longrightarrow \tupH^{1}(\Gamma_{0}(\cN), V_{k,k}^{\vee}) \\
 \mathrm{Res'} : \tupH^{1}(\Gamma_{0}'(\cM), V_{k,k}^{\vee}) &\longrightarrow \tupH^{1}(\Gamma_{0}(\cN), V_{k,k}^{\vee})
\end{align*}
induced by the inclusions $\Gamma_{0}(\cN) \subset \Gamma_{0}(\cM)$ and $\Gamma_{0}(\cN) \subset \Gamma_{0}'(\cM)$. (Note these are different restriction maps from $\mathrm{res}_\partial$ above). In fact, the restriction homomorphisms respect the Eisenstein/cuspidal decomposition, i.e. 
\[(\mathrm{Res}+\mathrm{Res'})(\tupH^{1}_{?}(\Gamma_{0}(\cM), V_{k,k}^{\vee}) \oplus \tupH^{1}_{?}(\Gamma_{0}'(\cM), V_{k,k}^{\vee})) \in \tupH^{1}_{?}(\Gamma_{0}(\cN), V_{k,k}^{\vee}) \]
for $? \in \{\mathrm{cusp},\mathrm{Eis}\}$ (compare \cite[\S4]{Gre09}). On the cuspidal part, by Theorem~\ref{thm:eichlershimuraharder}, we have
\begin{equation} \label{eqn:restrictionofcuspforms1}
\xymatrix@C=18mm{
\tupH^{1}_{\mathrm{cusp}}(\Gamma_{0}(\cM), V_{k,k}^{\vee}) \oplus \tupH^{1}_{\mathrm{cusp}}(\Gamma_{0}'(\cM), V_{k,k}^{\vee}) \ar[d]^{\cong} \ar[r]^-{(\mathrm{Res}+\mathrm{Res'})} & \tupH^{1}_{\mathrm{cusp}}(\Gamma_{0}(\cN), V_{k,k}^{\vee}) \ar[d]^{\cong}\\
 S_{k,k}(U_{0}(\cM)) \oplus S_{k,k}(U_{0}'(\cM)) \ar[r]^-{(\mathrm{Res}+\mathrm{Res'})} & S_{k,k}(U_{0}(\cN))}
\end{equation}
where $U_{0}'(\cM) := \alpha U_{0}(\cN)\alpha^{-1}$ and the bottom maps are the obvious inclusions. Note that conjugation by $\alpha$ induces a canonical isomorphism
\[ \tupH^{1}_{\mathrm{cusp}}(\Gamma_{0}(\cM), V_{k,k}^{\vee}) \cong \tupH^{1}_{\mathrm{cusp}}(\Gamma_{0}'(\cM), V_{k,k}^{\vee})\]
and correspondingly under Theorem~\ref{thm:eichlershimuraharder}, translation by $\alpha^{-1}$ induces an isomorphism
\begin{equation} \label{eqn:identificationofcuspforms}
S_{k,k}(U_{0}(\cM)) \cong S_{k,k}(U_{0}'(\cM)) 
\end{equation}
which we use to identify these two spaces. 
\begin{definition} \label{defn:p-old} The subspace of $S_{k,k}(U_{0}(\cN))$ of $\pri$-old forms is the image of the map 
\begin{align*}
S_{k,k}(U_{0}(\cM)) \oplus S_{k,k}(U_{0}(\cM)) & \xrightarrow{\mathrm{Res}+\mathrm{Res'}} S_{k,k}(U_{0}(\cN))\\
(f_{1}(g),f_{2}(g)) & \mapsto f_{1}(g) + f_{2}(g\alpha^{-1}).
\end{align*}
\end{definition} 
 Let $U_{0}(\cM) = \bigsqcup_{i \in I_{\pri}}U_{0}(\cN)\gamma_{i}$ be a system of coset representatives with $\gamma_{i} \in U_{0}(\cM)$.
\begin{definition}{\cite[Defn.~ 2.16]{BW19}} \label{defn:p-new}
The subspace of $S_{k,k}(U_{0}(\cN))$ of $\pri$-new cusp forms is the kernel of the map
\begin{equation} \label{eqn:p-new}
 S_{k,k}(U_{0}(\cN)) \xrightarrow{\phi_{s}\oplus\phi_{t}} S_{k,k}(U_{0}(\cM)) \oplus S_{k,k}(U_{0}(\cM))
\end{equation}
where $\phi_{s}$ and $\phi_{t}$ are the degeneracy maps 
\[
\phi_{s}(f)(g) := \sum\limits_{i \in I_{\pri}}f(g\gamma_{i}), \hspace{12pt}
\phi_{t}(f)(g) := \sum\limits_{i \in I_{\pri}}f(g\gamma_{i}\alpha).
\]
\end{definition}
\begin{lemma} \label{lem:IharaLemma}
The sum of the restriction maps
\[ \tupH^{1}_{\mathrm{cusp}}(\Gamma_{0}(\cM), V_{k,k}^{\vee}) \oplus \tupH^{1}_{\mathrm{cusp}}(\Gamma_{0}'(\cM), V_{k,k}^{\vee}) \xrightarrow{\mathrm{Res}+\mathrm{Res}'} \tupH^{1}_{\mathrm{cusp}}(\Gamma_{0}(\cN), V_{k,k}^{\vee})\]
is injective.
\end{lemma}
\begin{proof}
By (\ref{eqn:restrictionofcuspforms1}) and (\ref{eqn:identificationofcuspforms}), it suffices to show that 
\[ S_{k,k}(U_{0}(\cM)) \oplus S_{k,k}(U_{0}(\cM)) \xrightarrow{\mathrm{Res}+\mathrm{Res'}} S_{k,k}(U_{0}(\cN)) \]
is injective. The endomorphism
\begin{align*}
[S_{k,k}(U_{0}(\cM)) \oplus S_{k,k}(U_{0}(\cM))] & \xrightarrow{\mathrm{Res}+\mathrm{Res'}} S_{k,k}(U_{0}(\cN)) \xrightarrow{\phi_{s}\oplus\phi_{t}}  [S_{k,k}(U_{0}(\cM)) \oplus S_{k,k}(U_{0}(\cM))]   
\end{align*}
is given by the matrix
\begin{equation} \label{eqn:matrix} \begin{pmatrix} N(\pri)+1& N(\pri)^{-k/2}T_{\pri}\\N(\pri)^{-k/2}T_{\pri} & N(\pri)+1 \end{pmatrix} 
\end{equation}
(see \cite[Lemma 3.4.8]{CV19} for $k = 0$ and \cite[\S 5]{KK13} for $k > 0$). It is invertible since the determinant is bounded below by $2N(\pri)+1>0$ using the trivial bound known towards the Generalised Ramanujan--Petersson conjecture, obtained by Jacquet and Shalika in \cite{10.2307/2374103, 10.2307/2374050} (see also \cite{LRS99}). This shows that $(\mathrm{Res}+\mathrm{Res}')$ is injective.    
\end{proof}
We now look at the maps $\mathrm{Res}$ and $\mathrm{Res'}$ on the Eisenstein parts. Let $\{s_{1},\ldots,s_{m}\}$ be representatives of the cusps of $\Gamma_{0}(\cM)$. It can be checked that $\{s_{1}',\ldots,s_{m}'\}$, for $s_{i}' = \alpha\cdotspace s_{i}$, form a set of representatives for the cusps of $\Gamma_0'(\cM)$, and that moreover the union $\{s_{1},\ldots,s_{m},s'_{1},\ldots,s'_{m}\}$ is a set of representatives for the cusps of $\Gamma_{0}(\cN)$. By definition, we thus see that the map
\[
\h^1_\partial(\Gamma_0(\cM),V_{k,k}^\vee) \oplus \h^1_\partial(\Gamma_0'(\cM),V_{k,k}^\vee)  \xrightarrow{\mathrm{Res}+\mathrm{Res'}} \h^1_\partial(\Gamma_0(\cN),V_{k,k}^\vee)
\]
is an isomorphism; and then from the definition of Eisenstein cohomology, and \eqref{eqn:identifyingcohomology}, we deduce
\[ 
\tupH^{1}_{\mathrm{Eis}}(\Gamma_{0}(\cM),V_{k,k}^{\vee}) \oplus \tupH^{1}_{\mathrm{Eis}}(\Gamma_{0}'(\cM),V_{k,k}^{\vee}) \xrightarrow{\mathrm{Res}+\mathrm{Res'}} \tupH^{1}_{\mathrm{Eis}}(\Gamma_{0}(\cN),V_{k,k}^{\vee}) 
\]
is injective (compare \cite[Lem.\ 5]{Gre09}). We summarise as follows.
\begin{corollary} \label{cor:injectivitylevelchange} The map
\[ \tupH^{1}(\Gamma_{0}(\cM),V_{k,k}^{\vee}) \oplus \tupH^{1}(\Gamma_{0}'(\cM),V_{k,k}^{\vee}) \xrightarrow{\mathrm{Res}+\mathrm{Res'}} \tupH^{1}(\Gamma_{0}(\cN),V_{k,k}^{\vee}) \]
is injective.
\end{corollary}
\begin{proposition} \label{prop:appendixmayervietoris}
There is a Mayer--Vietoris exact sequence
\begin{multline} \label{eqn:mayervietors} 0 \rightarrow \tupH^{0}(\Gamma,V_{k,k}^{\vee}) \rightarrow [\tupH^{0}(\Gamma_{0}(\cM),V_{k,k}^{\vee}) \oplus \tupH^{0}(\Gamma_{0}'(\cM),V_{k,k}^{\vee})] \rightarrow \tupH^{0}(\Gamma_{0}(\cN),V_{k,k}^{\vee}) \xrightarrow{\epsilon_1}  \\ \tupH^{1}(\Gamma,V_{k,k}^{\vee}) \xrightarrow{\epsilon_2} [\tupH^{1}(\Gamma_{0}(\cM),V_{k,k}^{\vee})  \oplus  \tupH^{1}(\Gamma_{0}'(\cM),V_{k,k}^{\vee})] \xrightarrow{\mathrm{Res}+\mathrm{Res'}} \tupH^{1}(\Gamma_{0}(\cN),V_{k,k}^{\vee}) \rightarrow \cdots
\end{multline}
\end{proposition}
\begin{proof}
This is \cite[Prop.~ 13, Chapter II.2.8]{Ser80} applied to the $\Gamma$-module $M = V_{k,k}^{\vee}$.
\end{proof}
\begin{corollary} \label{cor:vanishingofH1} $\tupH_{1}(\Gamma, V_{k,k}) = 0$.
\end{corollary}
\begin{proof}
By the universal coefficient theorem, we have an identification
	\begin{equation} \label{eqn:uct2}
	\tupH_{1}(\Gamma,V_{k,k}) \cong \tupH^{1}(\Gamma,V_{k,k}^{\vee})^{\vee}, 
	\end{equation}
so it suffices to prove $\tupH^1(\Gamma,V_{k,k}^\vee) = 0$.
For $k=0$ we have that $V_{k,k}^{\vee}=L$ is trivial as a $\Gamma$-module. Thus $\tupH^{0}(\Gamma,V_{k,k}^{\vee}) = \tupH^{0}(\Gamma_{0}(\cM),V_{k,k}^{\vee}) = \tupH^{0}(\Gamma_{0}'(\cM),V_{k,k}^{\vee}) = \tupH^{0}(\Gamma_{0}(\cN),V_{k,k}^{\vee}) = L$. The vanishing of $\tupH^{1}(\Gamma,V_{k,k}^{\vee})$ now follows from the exactness of (\ref{eqn:mayervietors}). For $k > 0$, by Theorem~\ref{thm:eichlershimuraharder} above, $\tupH^{0}(\Gamma_{0}(\cN), V_{k,k}^{\vee}) = 0$. By the exactness of (\ref{eqn:mayervietors}) we have $\mathrm{ker}(\epsilon_2) = 0$. By Corollary~\ref{cor:injectivitylevelchange}, we know that $\mathrm{ker}(\mathrm{Res}\oplus\mathrm{Res'}) = 0$. But once again by the exactness of (\ref{eqn:mayervietors}), we have $\mathrm{Im}(\epsilon_2) = 0$ which implies the vanishing of $\tupH^{1}(\Gamma,V_{k,k}^{\vee})$. 
 This completes the proof of Theorem~\ref{thm:padicabeljacobiimageofdarmoncycles}.
\end{proof}

\bibliographystyle{alpha}
\bibliography{references}
\end{document}